\documentclass[12pt]{amsart}
\usepackage{amsmath,amsfonts,amsthm,amssymb,color}
\usepackage{fullpage}

\newcommand{\sech}{\text{ sech}}
\newcommand{\Es}{ l^2 L^\infty H^{s}}
\newcommand{\LEs}{l^2l^\infty L^2 H^{-s}}

\newcommand\R{{\mathbb R}}
\newcommand\Z{{\mathbb Z}}
 \newtheorem{theorem}{Theorem}
 
\renewcommand\P{\mathcal P}
\newcommand\p{\sigma}

 \newtheorem{remark}{Remark}[section]
 \newtheorem{lemma}[remark]{Lemma}
 \newtheorem{cor}[remark]{Corollary}
 \newtheorem{proposition}[remark]{Proposition}
 \newtheorem{definition}[remark]{Definition}

\begin{document}
\title{Energy and local energy
 bounds for the 1-d cubic NLS equation in $H^{-\frac14}$}

\author{Herbert Koch}
\address{Mathematisches Institut   \\ Universit\"at Bonn }

\author{ Daniel Tataru}
\address {Department of Mathematics \\
  University of California, Berkeley}

\begin{abstract}
  We consider the cubic Nonlinear Schr\"odinger Equation (NLS) in one
  space dimension, either focusing or defocusing.  We prove that the
  solutions satisfy a-priori local in time $H^{s}$ bounds in terms of
  the $H^s$ size of the initial data for $s \geq -\frac14$. This
  improves earlier results of Christ-Colliander-Tao~\cite{CCT} and of
  the authors \cite{MR2353092}. The new ingredients are a localization
  in space and local energy decay, which we hope to be of independent
  interest.
\end{abstract}
\maketitle

\section{Introduction}

We consider the cubic Nonlinear Schr\"odinger  equation (NLS) 
\begin{equation}
i u_t - u_{xx} \pm u |u|^2 = 0, \qquad u(0) = u_0,
\label{nls}\end{equation}
in one space dimension, either focusing or defocusing.  This problem
is invariant with respect to the scaling
\[
u(x,t) \to \lambda u(\lambda x,\lambda^2 t) 
\]
as is the Sobolev space $\dot H^{-\frac12}$, which one may view as the
critical Sobolev space.  The NLS equation \eqref{nls} is also
invariant under the Galilean transformation
\[ 
u(x,t) \to   e^{icx-ic^2 t}     u (x+2ct,t) 
\]
which corresponds to a shift in the frequency space.  However the
space $\dot H^{-\frac12}$ is not Galilean invariant.

This problem is globally well-posed for initial data $u_0 \in L^2$,
and the $L^2$ norm of the solution is conserved along the
flow. Furthermore, the solution has a Lipschitz dependence on the
initial data, uniformly for time in a compact set and data in bounded
sets in $L^2$.  Precisely, if $u$ and $v$ are two solutions for 
\eqref{nls} with initial data $u_0$, respectively $v_0$ then we have
\[
\| u(t) -v(t)\|_{L^2} \lesssim  \|u_0-v_0\|_{L^2}, \qquad |t| < 1, \quad
\|u_0\|_{L^2},\|v_0\|_{L^2} \leq 1
\]
By scaling and reiteration this implies a global in time bound
\begin{equation}
\| u(t) -v(t)\|_{L^2} \lesssim  e^{ C |t| (\|u_0\|_{L^2}+\|v_0\|_{L^2})^4}   
 \|u_0-v_0\|_{L^2}
\label{lip}\end{equation}

A natural question to ask is whether local well-posedness also holds
in negative Sobolev spaces between $H^{-1/2}$ and $L^2$.  As a
consequence of the Galilean invariance, the map from initial data to
the solution at time $1$ cannot be uniformly continuous in the unit
ball in $H^s$ with $s < 0$, (see \cite{MR1813239}, \cite{MR2018661}).
However, it is not implausible that one may have well-posedness with
only continuous dependence on the initial data. This problem seems to
be closely related to that of relaxing the exponential bound in
\eqref{lip} to a polynomial bound. Choosing the focusing or the
defocusing problem may also make a difference.

At this point we are unable to tackle the question of uniqueness or
continuous dependence on the initial data in $L^2$ in the $H^s$ norm
for $ s< 0$.  This remains a fundamental open problem, whose answer
may depend also on the focusing or defocusing character of the
equation.

The problem of obtaining apriori estimates in negative Sobolev spaces
was previously considered by Christ-Colliander-Tao~\cite{CCT} ($s \geq
-1/12$) and by the authors \cite{MR2353092}($s \geq -1/6$). One key
idea was that one can bootstrap suitable Strichartz type norms of the
solution but only on frequency dependent time-scales. Another idea was
to use the $I$-method to construct better almost conserved $H^s$ type
norms for the problem.

In this article we introduce another ingredient into the mix, namely 
local energy bounds. By establishing separately that the solutions 
satisfy local energy  bounds on the unit time scale we are able 
to weaken  the interval summation losses and obtain a better result
with more a-priori bounds on the solutions.

As in our previous work \cite{MR2353092}, here we focus on the
question of a-priori bounds in negative Sobolev spaces.  In the
process, we also establish certain space-time bounds for the solution,
as well as for the nonlinearity in the equation; these bounds insure
that the equation is satisfied in the sense of distributions even for
weak limits, and hence we also obtain existence of global weak
solutions for initial data in $H^s$ for $-1/4 \le s < 0$.  It is
likely that $-1/4$ is not optimal. Our main result is as follows:

\begin{theorem}\label{main} 
There exists $\varepsilon>0$  such that the following is true. Let
\[
 -\frac14 \le s < 0,  \qquad \Lambda\geq 1 
\]
and assume  that the initial data $u_0 \in L^2$ satisfies 
\[ 
\Vert u_0 \Vert_{H^s_\Lambda}^2 :=  \int (\Lambda^2+ \xi^2)^{s}  
 |\hat u_0|^2 d\xi < \varepsilon^2. 
\]
Then the solution $u$ to \eqref{nls} satisfies
\begin{equation}  
\sup_{0\le t \le 1} \Vert u(t) \Vert_{H^s_\Lambda} \le 2 \Vert u_0 
\Vert_{H^s_\Lambda}. 
\end{equation} 
\end{theorem}

As a byproduct of our analysis, in addition to the uniform bound
\eqref{mainxs}, we also establish space-time bounds for the solution
$u$ as well as for the nonlinearity $|u|^2 u$ in the time interval
$[0,1]$, namely
\begin{equation} \label{mainxs} \| \chi_{[0,1]} u\|_{X^s_\Lambda \cap
    X^s_{\Lambda,le}} \lesssim \| u_0\|_{H^s_\Lambda}, \qquad \|
  \chi_{[0,1]}|u|^2 u\|_{Y^s_\Lambda \cap Y^s_{\Lambda,le}} \lesssim
  \| u_0\|_{H^s_\Lambda}
\end{equation}
where the spaces $X^s_\Lambda$, $X^s_{\Lambda,le}$, $Y^s_\Lambda$ and  
$Y^s_{\Lambda,le}$ are defined in the next section.

The above theorem captures most of the technical contents of our
analysis.  However, it is not scale invariant, so taking scaling into
account we obtain further bounds. Indeed, rescaling 
\[
u_{\mu}(x,t) =
\mu u(\mu x,\mu^2 t)
\]
we have
\[
 \Vert u_\mu(0) \Vert_{H^s_{\mu \Lambda}} = \mu^{\frac{2s+1}2} \Vert
u(0) \Vert_{H^s_\Lambda}  
\]
Applying the above theorem to $u_\mu$ for $s=\frac14$ we obtain the  
case $s= \frac14$ of the following 

\begin{cor}
  Let $-\frac14\le s \le 0$.  Suppose that $M > 0$ and $\Lambda > 0$
  satisfy $ \Lambda \gg M^{4}$.  Let $u$ be a solution to
  \eqref{nls} with initial data $u_0 \in L^2$ so that
\[
\| u_0\|_{H^{-\frac14}_\Lambda} \leq M
\]
Then $u$  satisfies
\begin{equation}\label{14tos}
\sup_{|t| \le T} \Vert u(t) \Vert_{H^s_\Lambda} 
\lesssim \Vert u_0 \Vert_{H^s_\Lambda}, \qquad T \ll M^{-8}
\end{equation} 
\end{cor}
The general case follows from the $s = \frac14$ case due to the following
equivalence:
\begin{equation}
\|v\|_{H^s_\Lambda}^2 \approx \sum_{\lambda \geq \Lambda} 
\lambda^{\frac12+2s} \|v\|_{H^{-\frac14}_\lambda}^2
\end{equation}
Here and below all the $\lambda$ summations are dyadic.

Applying the above corollary to a given solution for increasing values
of $\Lambda$ yields global in time bounds. Consider first the case
when $1/4 < s < 0$.  Given $M \geq 1$ and an initial data $u_0$ so
that $\|u_0\|_{H^s} \leq M$ we have
\[
\| u_0\|_{H^{-\frac14}_\Lambda} \lesssim \Lambda^{-s-\frac14} M, \qquad 
\Lambda \geq 1
\]
By the above corollary this yields
\[
\sup_{0\le t \le T} \Vert u(t) \Vert_{H^{s}_\Lambda} 
\lesssim \Vert u_0 \Vert_{H^{s}_\Lambda}, \qquad 
\Lambda \gg \max\{ T^{\frac{1}{8s+2}} M^{\frac{4}{4s+1}}, M^{\frac{2}{2s+1}}\}
\]
Hence we have proved

\begin{cor}
  Let $-\frac14 < s < 0$ and $M \geq 1$. Let $u$ be a solution  to
  \eqref{nls} with initial data $u_0 \in L^2$ so that
\[
\| u_0\|_{H^s} \leq M
\]
Then for all $T > 0$ the function $u$  satisfies
\begin{equation}  
 \sup_{|t| \le T} \Vert u(t) \Vert_{H^s_{\Lambda(T)}} \lesssim  M, 
  \qquad
 \Lambda(T) =\max\{ T^{\frac{1}{8s+2}} M^{\frac{4}{4s+1}}, M^{\frac{2}{2s+1}}\}
\end{equation} 
\end{cor}

Here it is only the principle that matters.  The exact exponents here
are less important since it is very unlikely that the $s = -\frac14$
result is sharp.

  The case $s = -\frac14$ is more delicate. There all
we can say is that
\[ 
\lim_{\Lambda \to \infty} \Vert u_0 \Vert_{H^{-\frac14}_\Lambda} =0 \quad \text{ for
} u_0 \in H^{-\frac14}. 
\]
Thus we obtain
\begin{cor}
  Let $u$ be a solution  to
  \eqref{nls} with initial data $u_0 \in L^2$. 
Then for all $T > 0$ the function $u$  satisfies
\begin{equation}  
\sup_{|t| \le T}  \Vert u(t) \Vert_{H^{-\frac14}_{\Lambda(T)}} \leq 1 
\end{equation} 
for some increasing function $\Lambda(T)$ which only depends on the 
the $H^{-\frac14}$ frequency envelope of $u_0$.
\end{cor}

The apriori estimates suffice to construct global weak solutions. 
Using the uniform bounds \eqref{mainxs} one may prove the following statement. 

\begin{theorem} Suppose that $u_0 \in H^s$, $s\ge -\frac14$. Then there 
exists a weak solution $u \in C(\mathbb{R}, H^s) $,
so that for all $T > 0$ we have\footnote{Here we drop the subscript $\Lambda = 1$ from the notation for the space-time norms.} $ \chi_T  u \in X^s \cap X^s_{le}$ and 
\begin{equation} 
\sup_{-T \le t \le T } \Vert u(t) \Vert_{H^s} + \Vert \chi_{[-T,T]}u  \Vert_{X^s\cap X^s_{le}} 
+ \Vert \chi_{[-T,T]} |u|^2 u \Vert_{Y^s\cap Y^s_{le}} 
\le C 
\end{equation} 
with $C$ depending on $T$ and on the $H^{-\frac14}$ 
frequency  envelope of $u_0$. 
\label{main3} 
\end{theorem}

\subsection{Some heuristic considerations} 
The nonlinear Schr\"odinger equation is completely
integrable. Depending on whether we look at the focusing or the
defocusing problem, we expect two possible types of behavior for
frequency localized data. 

In the defocusing case, we expect the solutions to disperse
spatially. However, in frequency there should only be a limited
spreading, to a range below the dyadic scale, which depends only on
the $L^2$ size of the data. Precisely energy estimates show that for
frequency localized data with $L^2$ norm $\lambda$, frequency
spreading occurs at most up to scale $\lambda$.

In the focusing case, the expected long time behavior (or short time
for large data) is a resolution into a number of solitons (possibly
infinitely many) plus a dispersive part. The situation is somewhat
complicated by the fact that some of these solitons may have the same
speed, and thus considerable overlapping. The inverse scattering
formalism provides formulas for such solutions with many interacting
solitons.  Nevertheless it is instructive to consider first the case
of a single soliton, which in the simplest case has the form
\[   
u(x,t) = e^{-it} \sech (2^{-1/2} x).  
\]
Rescaling we get a soliton with $L^2$ norm $\lambda$, namely 
\[ 
  u^{\lambda} (x,t)= e^{-it\lambda^4} \lambda^2 \sech ( 2^{-1/2} \lambda^2 x ).   
\]
More soliton solutions can be obtained due to the Galilean
invariance. However, our function spaces here break the Galilean
invariance, so our worst enemies are the zero speed solitons.

The above solution is constant in time, up to a phase factor. It is
essentially localized to an interval of size $\lambda^{-2}$ in $x$, and of
size $\lambda^2$ in frequency. It also saturates our local energy estimates
in \eqref{mainxs} for $s = -\frac14$, exactly when $\Lambda = \lambda^4$.

In many cases error estimates for a   nonlinear semiclassical ansatz 
 for solutions  are available. An example is the initial data
\[
 u_0(x) = \lambda \sech(2^{-1/2} x) 
\]
where a semiclassical ansatz for an approximate solution is given by 
\[
 u(t,x) = \lambda A(x,t) e^{-i \lambda S(t,x) } 
\] 
where $\rho = A^2$ and  $\mu = A^2 \partial_x S$ satisfy the Whitham equations
\[ 
\rho_t + \lambda\partial_x \mu =0 \qquad \partial_t \mu + 
\lambda \partial_x (\mu^2/\rho \pm \rho^2/2) = 0
\]
with $+$ in defocusing and $-$ in the focusing case. The Whitham
equations are hyperbolic for the defocusing case and they can be
solved up to an time $T \sim \lambda^{-1}$, when singularities
corresponding to caustics occur.  Grenier \cite{MR1425123} has
justified this ansatz up to the time when caustics occur.

The Whitham equations are elliptic in the focusing case. Akhmanov,
Khokhlov and Sukhorukov \cite{AKS} realized that the implicit
equation
\begin{equation} \label{AKS}
\mu = -2 \lambda t \rho^2 \tanh\left( \frac{\rho x- \mu \lambda
    t}{\rho} \right), \qquad \rho = (1+ \lambda^2 t^2 \rho^2) \sech^2
\left( \rho x-\mu t \right) 
\end{equation}
defines a solution to the Whitham equation with the $\lambda \sech$
initial data.  The semiclassical ansatz for small semiclassical times
has been studied by Thomann \cite{MR2422717}.

The direct scattering problem has been solved by Satsuma and Yajima
\cite{MR0463733}. In particular, if $\lambda$ is an integer one
obtains a pure soliton solution with $\lambda$ solitons with velocity
$0$.  In this case the solution is periodic with period
$2$. Formula \eqref{AKS} seems to indicate that
the solution remains concentrated in an spatial area for size $\sim
\ln (1+\lambda)$.
The semiclassical
limit has been worked in a number of problems, see  Jin, Levermore and McLaughlin
\cite{MR1670048},  Kamvissis \cite{MR1751356},
Deift and Zhou \cite{MR1207209}.
and  Kamvissis, McLaughlin and Miller \cite{MR1999840}.

 These examples indicate that energy may spread over a large frequency interval even if the energy 
is concentrated at frequencies  $\lesssim 1$ initially, and there are solutions with 
energy distributed over a large  frequency interval with velocity zero. 
For the proof of our main result we use localization in frequency and space. These examples provide natural 
limits for the localization. This is reflected in the estimates and the definition of the function spaces.

\subsection{An overview of the proof}  We begin with a
dyadic Littlewood-Paley frequency decomposition of the solution $u$,
\[
u = \sum_{\lambda \geq \Lambda} u_\lambda, \qquad u_\lambda = P_\lambda u
\]
where $\lambda$ takes dyadic values not smaller than $\Lambda$, and
$u_\Lambda$ contains all frequencies up to size $\Lambda$.  Here the
multipliers $P_\lambda$ are standard Littlewood-Paley projectors. For
each such $\lambda$ we also use a spatial partition of unity on the
$\lambda^{1+4s}$ scale,
\begin{equation}\label{chij}
1 = \sum_{j \in \Z} \chi_j^\lambda(x), \qquad \chi_j^\lambda(x)=
\chi(\lambda^{-1-4s} x - j)
\end{equation}
with $\chi \in C^\infty_0(-1,1)$.
To prove the theorem we will use


\begin{itemize}
\item[(i)]  Two energy spaces, namely  a standard energy norm
\begin{equation}\label{edef} 
\| u \|_{\Es_\Lambda}^2 =\sum_{\lambda \geq \Lambda} \lambda^{2s}  
\| u_{\lambda}\|_{L^\infty L^2}^2
\end{equation}
and a local energy norm\footnote{For $s = -\frac14$ the spatial scale
  is one and this corresponds to the familiar gain of one half of a
  derivative. It may seem more natural to remove the $\partial_x$
  derivative and appropriately adjust the power of $\lambda$. This
  would be equivalent for all frequencies $\lambda >
  \Lambda$. However, in $u_\Lambda$ we are including all lower
  frequencies, which correspond to waves with lower group velocities
  and to a worse local energy bound, should the operator $\partial_x$
  not be present here. Based on the standard form of the local energy
  bounds for the linear Schr\"odinger equation one may still expect to
  be able to relax the $\partial_x$ operator almost to
  $\partial_x^\frac12$.  At least in the focusing case this is not
  possible; indeed, if $s = -\frac14$ then the local energy component
  of the bound \eqref{le-nonlin} below is saturated by the frequency
  $\Lambda^\frac12$ soliton.  }  adapted to the $\lambda^{1+4s}$
spatial scale,
\begin{equation}\label{ledef}
 \| u \|_{\LEs_\Lambda}^2 =
 \sum_{\lambda \geq \Lambda} \lambda^{-2s-2}   \sup_{j \in \Z}
\|\chi_j^\lambda \partial_x u_{\lambda}\|_{ L^2}^2
\end{equation}

\item[(ii)]  Two Banach spaces $X^s_\Lambda$ and $X^s_{\Lambda,le}$  
measuring the space-time regularity of the solution $u$.  The first one 
measures the dyadic parts of $u$ on small frequency dependent 
timescales, and is mostly similar to the spaces introduced in 
\cite{CCT}, \cite{MR2353092}. The second one is new, and measures
the spatially localized size of the solution on the unit time scale.
These spaces are  defined in the next section.

\item[(iii)] Two corresponding Banach spaces $Y^s_\Lambda$ and
  $Y^s_{\Lambda,le}$ measuring the regularity of the nonlinear term
  $|u|^2 u$.  These are also defined in the next section.
\end{itemize}

\noindent
The linear part of the argument is a straightforward consequence 
of our definition of the spaces, and is given by 
\begin{proposition}
The following estimates hold for solutions to \eqref{nls}:
\begin{equation}
\| u\|_{X^s_\Lambda} \lesssim  \| u \|_{\Es_\Lambda} + 
\| (i\partial_t -\Delta) u\|_{Y^s_\Lambda}
\label{en-lin}\end{equation}
respectively
\begin{equation}
\| u\|_{X_{\Lambda,le}^s} \lesssim  \| u \|_{\LEs_\Lambda} + 
\| (i\partial_t -\Delta) u\|_{Y_{\Lambda,le}^s}
\label{le-lin}\end{equation}
\label{plin}\end{proposition}

\noindent
To estimate the nonlinearity we need a cubic bound,

\begin{proposition}
Let $u \in X^s_\Lambda \cap X^s_{\Lambda,le}$. 
Then $|u|^2 u \in Y^s_\Lambda \cap Y^s_{\Lambda,le}$ and
\begin{equation}
\| |u|^2 u\|_{Y^s_\Lambda \cap Y^s_{\Lambda,le}} \lesssim 
\|u\|_{X^s_\Lambda\cap X^s_{\Lambda,le}}^3
\end{equation}
\label{pnonlin}\end{proposition}

\noindent
Finally, to close the argument we need to propagate the energy norms:

\begin{proposition}
Let  $u$ be a solution to \eqref{nls} with
\[ 
\| u\|_{l^2 L^\infty H^{s}_\Lambda} \ll 1. 
\]
Then  we have the energy bound
\begin{equation}
\| u\|_{l^2 L^\infty H^{s}_\Lambda} \lesssim \| u_0\|_{H^{s}_\Lambda} 
  + \|u\|_{X^s_\Lambda \cap X^s_{\Lambda,le}}^3,
\label{en-nonlin}\end{equation}
respectively the local energy decay 
\begin{equation}
\| u\|_{\LEs_\Lambda} \lesssim \| u_0\|_{H^{s}_\Lambda} 
  + \|u\|_{X^s_\Lambda \cap X^s_{\Lambda,le}}^3.
\label{le-nonlin}
\end{equation}
\label{penergy}\end{proposition}
The bootstrap argument which leads from Propositions~\ref{plin},\ref{pnonlin}
and \ref{penergy} to Theorem~\ref{main} is straightforward and thus omitted.
Instead we refer the reader to the similar argument in \cite{MR2353092}.

We remark that in our set-up $s =-\frac14$ is the actual threshold
in the energy estimates in Proposition \ref{penergy}, though not in
the bounds for the cubic nonlinearity in Proposition \ref{pnonlin}.
In principle the former can be improved by adding further corrections
to the energy functional; we choose not to pursue this here.

The plan of the paper is as follows. In the next section we motivate
and introduce the spaces $X^s_\Lambda$, $X^s_{\Lambda,le}$,
$Y^s_\Lambda$ and $Y^s_{\Lambda,le}$, as well as establish the linear
mapping properties in Proposition~\ref{plin}. In Section~\ref{sbilin}
we discuss the linear and bilinear Strichartz estimates for solutions
to the linear equation.

The trilinear estimate in Proposition~\ref{pnonlin} is proved in
Section~\ref{snonlin}. In the last section we use a variation of the
I-method to construct a quasi-conserved energy functional and compute
its behavior along the flow, thus proving 
the first bound \eqref{en-nonlin} in Proposition~\ref{penergy}. A
modification of the same idea leads to the local energy decay estimate
\eqref{le-nonlin}.

\section{ The function spaces}
To understand what to expect in terms of the regularity of $u$ we
begin with some heuristic considerations.  If the initial data $u_0$
to \eqref{nls} satisfies $\|u_0\|_{L^2} \ll 1$ then the equation can
be solved iteratively using the Strichartz estimates on a unit time
interval.  We obtain essentially linear dynamics, by which we mean that the difference between the solution to the linear Schr\"odinger equation 
and NLS is small, and the solution $u$
belongs to the space $X^{0,1}[0,1]$ associated to the Schr\"odinger
equation (see the definition in \eqref{X01} below).

Let $s < 0$. Consider now NLS with initial data $u_0 \in H^{s}$,
localized at frequency $\lambda$. Then the initial data satisfies
$\|u_0\|_{L^2} \lesssim \lambda^{-s}$. By rescaling the small $L^2$
data result we conclude that the evolution is still described by
linear dynamics up to the shorter time $\lambda^{4s}$.

We expect the frequency localization of the solution to be somewhat
robust.  Then it is natural to consider a decomposition of the
solution $u$ into its dyadic components $u_\lambda = P_\lambda u$ and
to measure the $u_\lambda$ component uniformly in $\lambda^{4s}$ time
intervals.

Linear waves with frequency $\lambda$ travel with group velocity  $2\lambda$,
therefore they cover a distance of about $\lambda^{1+4s}$ within a
$\lambda^{4s}$ time interval. Hence we can naturally partition
frequency $\lambda$ waves with respect to a grid of size
\[
\delta t_\lambda = \lambda^{4s}, \qquad \delta x_\lambda = \lambda^{1+4s}
\]

Correspondingly we have the spatial partition of unity \eqref{chij}.
We remark that the scale of this partition increases with $\lambda$
for $s > - 1/4$, and decreases with $s$.  It is independent of
$\lambda$ exactly for $s = -1/4$, which makes the  threshold $s =-\frac14$ 
 very convenient technically.

Now we  introduce the function spaces for the solutions
$u$.  Following an idea of M. Christ, given an interval $I=[t_0,t_1]$ 
we define  the space 
\begin{equation}
 \| \phi\|_{X^{0,1}[I]}^2 = \|\phi(t_0)\|_{L^2}^2 +
 |I|    \| (i \partial_t-\Delta) \phi\|_{L^2[I]}^2
\label{X01}\end{equation}
Ideally we would like to place the dyadic pieces $u_\lambda$ of $u$ in
such a space on the $\delta t_\lambda$ scale. However, this does not
quite work and we have to introduce a slightly larger space
\[
X_\lambda [I] = X^{0,1}[I]+\lambda^{-1-2s} U^2_\Delta[I].
\]
The $U^p$ and $V^p$ spaces are a refinement of the Fourier restriction
spaces of Bourgain. We refer to the next section and to
\cite{MR2094851}, \cite{MR2526409} for a discussion of them. They are
related to the Bourgain spaces through the embeddings
\begin{equation}
X^{0,\frac12,1} \subset U^2_\Delta \subset  X^{0,\frac12,\infty}
\label{embed}\end{equation}
where the above $X^{s,b}$ type norms are defined by
\begin{equation} \label{XSB} 
\Vert u \Vert_{X^{0,\frac12,1}} = \sum_{\mu } 
\mu^{\frac12} \Vert Q_\mu u \Vert_{L^2},
\qquad  \Vert u \Vert_{X^{0,\frac12,\infty}} = \sup_{\mu } 
\mu^{\frac12} \Vert Q_\mu u \Vert_{L^2}. 
\end{equation} 
and the modulation localization multipliers $Q_\mu$ select the dyadic
region $ \{|\tau+\xi^2| \sim \mu\}$.

Let us compare the two parts of the $X_\lambda$ norms.
First, H\"older's inequality implies 
\[ 
\| (i \partial_t-\Delta) \phi\|_{L^1(I, L^2)} \le |I|^{\frac12} \| (i \partial_t-\Delta) \phi\|_{L^2[I]}
\]
and hence (again referring to the next section for a discussion of the
$U^p$ spaces)
\[ 
\Vert u \Vert_{U^1_\Delta(I)} \le   \Vert   u \Vert_{X^{0,1}(I)} 
\]
and by the embedding properties of the $U^p$ spaces
\[ 
\Vert u_\lambda \Vert_{U^2_\Delta(I)} \le  \Vert u_\lambda  \Vert_{X^{0,1}(I)}. 
\]
Thus we obtain 
\begin{equation}
\label{em} 
 \Vert u_\lambda \Vert_{U^2_\Delta(I)} \lesssim \|u_\lambda  \|_{X_{\lambda}[I]}
\end{equation} 
This bound will suffice for most of our estimates.

On the other hand, the structure of the $X_\lambda$ norms is so that 
we expect to have better bounds at high modulations ($\gtrsim \lambda^2$, e.g.).
However, some care is required in order to make this precise, because
modulation localizations do not commute with interval localizations.
To address this issue we introduce extension operators $E_I$ which 
take a function $u \in X_\lambda[I]$ to its extension $E_I u$ solving 
the homogeneous Schr\"odinger equation outside $I$  with matching
data at the two endpoints of $I$. By definition we have
\[
\| (i \partial_t -\Delta) E_I u_\lambda\|_{\lambda^{-2s} L^2 + \lambda^{-1-2s}
DU_\Delta^2} \lesssim \| u_\lambda\|_{X_\lambda[I]}, \qquad |I| = \lambda^{4s}
\]
This implies the high modulation bound
\begin{equation}
\label{em2} 
   \Vert Q_{\geq \sigma} E_I u_\lambda  \Vert_{L^2} 
\lesssim \min\{\lambda^{-2s} \sigma^{-1}, \lambda^{-1-2s} \sigma^{-\frac12}\}
 \Vert u_\lambda \Vert_{X_\lambda[I]}. 
\end{equation} 
 We remark that the
balance between the norms of the two component spaces in $X_\lambda$
is achieved at modulation $\lambda^2$. Since the $X_\lambda$ norm is
only used on frequency $\lambda$ functions, it follows that the
$U^2_\Delta$ component of the space $X_\lambda$ is only relevant in
the elliptic region $\{|\tau - \xi^2| \sim |\tau|+|\xi|^2\}$.

Now we can define the $X^s_\Lambda$ norm in a time interval $I$ by
\begin{equation}
\| u\|_{X^s_\Lambda[I]}^2 = \sum_{\lambda \geq \Lambda} \lambda^{2s}  
\sup_{|J| = \lambda^{4s}, 
J\subset I}
\| u_\lambda\|_{X_\lambda[J]}^2  
\label{xs}\end{equation}
In the sequel we will mostly drop the interval $I=[0,1]$ from the notation. 
We remark that within each interval $J$ we have square summability 
on the $ \lambda^{1+4s}$ spatial scale as well as on any larger scale,
\begin{equation}
   \sum_{j \in \Z}  
\| \chi^\mu_j u_\lambda\|_{X_\lambda[J]}^2 \lesssim \| u_\lambda\|_{X_\lambda[J]}^2, 
\quad |J| = \lambda^{4s}, \quad \mu \gtrsim \lambda,
\label{xs-sq-sum}\end{equation}
This can be viewed as a consequence of the fact that frequency
$\lambda$ waves travel with speed $\lambda$.  We refer to
Lemma~\ref{l:xs-sq-sum} in the next section for more details.

Next we introduce the related local smoothing space
$X_{\Lambda,le}^s$, where the above summation with respect to spatial
intervals is replaced by a summation with respect to time intervals:
\begin{equation}
 \|u\|_{X_{\Lambda,le}^s[I]}^2 =\sum_{\lambda} \lambda^{2s-2} \sup_{j \in \Z} 
\sum_{  J \subset I}^{|J| = \lambda^{4s}}
\|\chi^\lambda_j \partial_x u_\lambda\|_{X_{\lambda}[J]}^2
\label{xse}\end{equation}
Here and below the $J$ summation is understood to be over a partition of $I$
into intervals $J$ of the indicated size.

To measure the regularity of the nonlinear term we begin with 
\[
Y_\lambda[I] = |I|^{-\frac12} L^2 + \lambda^{-1-2s} DU_\Delta^2[I]
\]
which is exactly the output of the linear Schr\"odinger operator
$i \partial_t -\Delta$ applied to $X_\lambda[I]$ functions
(see next section for a discussion of $DU^2$).  We use it to define 
the $Y^s_\Lambda$ norm 
\begin{equation}
\| f\|_{Y^s_\Lambda[I]}^2 = \sum_{\lambda \geq \Lambda} 
\lambda^{2s}   \sup_{|J| = \lambda^{4s}, \ J \subset I}
\|\chi_J f_\lambda\|_{Y_\lambda}^2  
\label{ys}\end{equation}
as well as its local energy counterpart
\begin{equation}
 \|f\|_{Y_{\Lambda,le}^s[I]}^2 =\sum_{\lambda \geq \Lambda} \lambda^{2s-2} \sup_{j \in \Z}
 \sum_{  J \subset I}^{|J| = \lambda^{4s}}
 \| \chi_J \chi^\lambda_j \partial_x f_\lambda\|_{Y_\lambda}^2
\label{yse}\end{equation}

\section{$U^p$ and $V^p$ spaces}

We sketch the construction of the spaces $U^p$ and $V^p$ and their
properties and refer to \cite{MR2094851}, \cite{MR2526409} for more
details. 

Both $U^p$ and $V^p$ are spaces of functions in $\R$ which take values
in a Hilbert space, $L^2(\R)$ in our case.  To define them we first
introduce the class $\P$ of finite partitions of $\R$ into intervals.
A partition $\p \in \P$ is determined by the endpoints of the
intervals, which are identified with a finite increasing sequence
$(t_n)_{n=0,N(\p)}$ with $t_0 = -\infty$ and $t_N{(\p)}=
\infty$.

Let $1\le p <\infty$. A $U^p$ atom is a right continuous 
piecewise constant function 
\[
a = \sum_{n=2}^{N(\p)}  1_{[t_{n-1},t_{n})} a_n, \qquad \sum \|a_n\|^p = 1, \quad 
\]
on the real line associated to a partition $\p = (t_n) \in \P$ of the real
line.
We define $U^p$ as the atomic space consisting of all functions for which the 
following norm is finite:
\[ 
\Vert u \Vert_{U^p} = \inf \left\{ \sum _{k=1}^\infty \lambda_k : f = \sum \lambda_k a_k, a_k
\text{ atoms }, \lambda_k \ge 0   \right\}
\]
This is a Banach space of bounded right continuous functions  which 
have limit zero as $t$ goes to $-\infty$.

The space of bounded $p$ variation functions $V^p$ consists of all functions
 on $\mathbb{R}$  for which the following norm is finite, 
\[ 
\Vert u \Vert_{V^p}^p = \sup_{\p \in \P} \sum_{n=2}^{N-1}
\Vert u(t_{n})-u(t_{n-1}) \Vert^p 
\]
In this formula we set $u(\infty) = 0$.  This is a Banach space of
bounded functions.  The functions in $V^p$ have lateral limits
everywhere.  By $V^p_{rc}$ we denote the subspace of right continuous
functions in $V^p$ which have limit zero as $t$ goes to $\infty$.

Both spaces are invariant under monotone reparametrizations of $\mathbb{R}$ 
and can therefore be easily defined for intervals.  Given any partition 
$\p=(t_n) \in \P$ of the real line we also have the interval summability bounds
\begin{equation}
\| u\|_{U^p}^p \lesssim \sum_{n=1}^{N(\p)} \|1_{[t_{n-1},t_n)} u\|_{U^p}^p.
\label{isum}\end{equation}

Clearly, if  $1\le p < q$ then
\[ 
U^p \subset U^q,\quad  V^p \subset V^q ,\quad  U^p \subset V^p
\]
We also have
the nontrivial relation
\[
V^p_{rc} \subset U^q \qquad 1\le p < q
\]
More precisely, as proved in \cite{MR2526409}, there exists $\delta>0$
such that for each $v \in V^p_{rc}$ and $M>1$ there exists $u\in U^p$,
$w \in U^q$  such that $v = u+w$ and
\begin{equation} 
 M^{-1} \Vert u \Vert_{U^p} + e^{\delta M} \Vert w \Vert_{U^q} \lesssim \Vert v
\Vert_{V^p_{rc}}. 
\end{equation} 

 The relation to of $U^p$ and $V^p$ to Besov spaces is as follows: 
\begin{equation} \label{besov}
  \dot B^{1/p,p}_1 \subset U^p \subset V^p_{rc} \subset \dot B^{1/p,p}_\infty. 
\end{equation} 
In particular the norms of $u$ in $U^p$ and $V^p_{rc}$ are equivalent
if $\hat u$ is supported in a fixed dyadic frequency interval.
Moreover if $Q_\mu$ denotes the projection to a dyadic frequency range
we have
\[ 
\Vert Q_\mu u \Vert_{L^p} \le c \mu^{-1/p} \Vert u \Vert_{V^p}. 
\]

There is also a duality relation: Let $1 < p,q < \infty$ be dual
exponents. Then
\begin{equation}  U^p \times V^q \ni (u,v) \to B(u,v)= \int u v_t dt 
\label{dual} 
\end{equation} 
defines an isometry $V^q \to (U^p)^* $.  The notation in \eqref{dual}
is formal, and making it rigorous requires considerable care, for
which we refer to \cite{MR2526409}.  We use the spaces $DU^p$ and
$DV^p_{rc}$ as distributional time derivatives of functions in $U^p$ and
$DV^p_{rc}$. This is possible since for $1/p+1/q=1$
\[ \Vert u \Vert_{U^p} = \sup \{ B(u,v): v \in C^\infty_0: \Vert v
\Vert_{V^q}=1\} \]
and 
\[ \Vert v \Vert_{V^q_{rc}} = \sup \{ B(u,v): u \in C^\infty_0: \Vert u
\Vert_{U^p}=1\}. \]
 
All these constructions apply to functions with
values in Hilbert spaces. Of particular interest is the Hilbert space 
$l^2$. A short reflection shows that 
\begin{equation} \label{ortho} 
 U^2(l^2)\subset l^2 U^2, 
\end{equation}
where on the left 
we have $l^2$ sequences with values in $U^2$, and on the right $l^2$ valued 
functions in $U^2$. Similarly $l^2 V^2 \subset  V^2l^2$.

We use Bourgain's recipe to adapt the function spaces to the 
Schr\"odin\-ger equation
\[ 
\Vert u \Vert_{U^p_{\Delta}} = \Vert e^{-it\Delta } u(t) \Vert_{U^p} 
\]
and
\[ 
\Vert v \Vert_{V^p_{\Delta}} = \Vert e^{-it\Delta } v(t) \Vert_{V^p}. 
\]
We will always consider right continuous functions and we drop 
$rc$ from the notation.

The relation to the $X^{s,b}$ spaces can be seen from 
estimate \eqref{besov}, which also implies the high modulation estimate
\[ 
\Vert Q_\mu u \Vert_{L^2} \le c \mu^{-1/2} \Vert u \Vert_{V^2_\Delta} 
\]
where, as before, $Q_\mu$ is the projection to modulations of size
$\mu$, namely the frequency region $\{ \tau +\xi^2 \approx \mu \}$.

The next lemma, combined with a rescaling argument, proves 
the bound \eqref{xs-sq-sum}. 

\begin{lemma}\label{l:xs-sq-sum}
  Let $\lambda >0$, and $I$ an interval with $ |I| \lambda \leq 1
  $.  If $u$ is frequency localized in $ [-\lambda, \lambda]$ then
  the following estimates hold\footnote{Here $\chi_j = \chi_j^1$; the superscript $1$ is omitted.} :
\begin{equation}
 \sum \Vert \chi_j  u \Vert_{U^2_\Delta[I]}^2 \lesssim 
\Vert u \Vert_{U^2_\Delta[I]}^2, 
\qquad  \sum \Vert \chi_j  f \Vert_{DU^2_\Delta[I]}^2 \lesssim 
\Vert f \Vert_{DU^2_\Delta[I]}^2, 
 \label{ixsum}\end{equation}
 \begin{equation}
 \Vert u \Vert_{V^2_\Delta[I]}^2 \lesssim \sum \Vert \chi_j u \Vert_{V^2_\Delta[I]}^2 .
\qquad 
 \Vert f \Vert_{DV^2_\Delta[I]}^2 \lesssim \sum \Vert \chi_j f \Vert_{DV^2_\Delta[I]}^2 .
 \label{dixsum}\end{equation} 
\end{lemma}

\begin{proof}
  It suffices to verify the first inequality for $U^p$  atoms. Furthermore,
due to the first bound in \eqref{isum}, we only need to prove it for each step in
an atom. Thus consider a solution $u$ to the homogeneous Schr\"odinger 
equation in a subinterval $J=[a,b) \subset I$. For each $j \in \Z$ we write
an equation for $\chi_j u$, namely
\[
(i \partial_t - \partial_x^2)(\chi_j u)= f_j 
\]
where the right hand sides $f_j$ are given by
\begin{equation}
f_j = - \partial_x^2 \chi_j u - 2 \partial_x
\chi_j \partial_x u 
\label{deffj}\end{equation}
and can be estimated as follows:
\[
\begin{split}
\sum_{j} \|f_j\|_{L^1_t L^2_x}^2
 \lesssim   |J| \sum_{j} \|f_j\|_{ L^2_{t,x}}^2 
\lesssim |J|(\| u\|_{L^2_{t,x}}^2+ \| \partial_x u\|_{L^2_{t,x}}^2)
 \lesssim   \lambda^2 |J|^2 \|u\|_{L^\infty_t L^2_x}^2 
\end{split}
\]
Then, using $\lambda |J| \leq 1$,  we have 
\[
\sum_j \| \chi_j u\|_{U^2[J]}^2 \lesssim \sum_j \| \chi_j u\|_{U^1[J]}^2
\lesssim \sum_j \| \chi_j u(a)\|_{L^2}^2 +  \|f_j\|_{L^1_t L^2_x}^2
\lesssim \|u(a)\|_{L^2}^2
\]
The proof of the first bound in \eqref{ixsum} is completed by summing
over the intervals $J$.

For the second bound we consider $f$ of the form $f = (i \partial_t
- \partial_x^2)u$ with $u \in U^2_\Delta[I]$. Then we can write
$\chi_j f$ as
\[
\chi_j f = (i \partial_t - \partial_x^2) (\chi_j u) - f_j 
\]
with $f_j$ as in \eqref{deffj}. For the first term we use the first bound in 
\eqref{ixsum}, and for the second we bound $f_j$ in $L^1 L^2$ as above.

Next we consider the first bound in \eqref{dixsum}. For a partition $\sigma = (t_n)$ 
of the interval $I$ we need to estimate the sum:
\[
S = \sum_n \|e^{i(t_{n+1} -t_n) \Delta} u(t_n) - u(t_{n+1})\|_{L^2}^2 
\]
We have
\[
\begin{split}
S \approx  \ & \sum_j  \sum_n  \|\chi_j(e^{i(t_{n+1} -t_n) \Delta} u(t_n) - u(t_{n+1}))\|_{L^2}^2 
\\ \lesssim  & \sum_j \sum_n  \|e^{i(t_{n+1} -t_n) \Delta}
(\chi_j u(t_n)) - \chi_j u(t_{n+1})\|_{L^2}^2 
 +  \sum_j \sum_n  \|[ \chi_j, e^{i(t_{n+1} -t_n) \Delta}] u(t_n) \|_{L^2}^2 
\end{split}
\]
and the first term is directly estimated in terms of the right hand side in 
\eqref{dixsum}. For the second term we will establish a stronger bound,
namely
\[
 \sum_j  \|[ \chi_j, e^{i T \Delta}] u_0 \|_{L^2}^2 \lesssim (T \lambda)^2 \|u_0\|_{L^2}^2
\]
whenever $u_0$ is frequency localized in $[-\lambda,\lambda]$.
This suffices since we have $\sum (t_{n+1} -t_n)^2 \lesssim |I|^2$.
Denoting $u(t) = e^{it \Delta} u_0$ we write
\[
(i \partial_t - \partial_x^2)[ \chi_j, e^{i T \Delta}] u_0 = f_j
\]
with $f_j$ again as in \eqref{deffj}. Hence
\[
\sum_j  \|[ \chi_j, e^{i T \Delta}] u_0 \|_{L^2}^2 \lesssim 
\sum_j \|f_j\|_{L^1L^2}^2
\]
and the bound for $f_j$ in $L^1 L^2$ is the same as above.
The proof of the first part of \eqref{dixsum} is concluded.

For the second part of \eqref{dixsum} we write $f$ in the form 
$f = (i \partial_t - \partial_x^2)u$ in $I = [a,b]$ with $u(a)= 0$.
Arguing by duality (see \eqref{dual}), from the first part of \eqref{ixsum}
applied to solutions for the homogeneous equation
we obtain the uniform energy bound 
\[
\|u\|_{L^\infty L^2}^2 \lesssim  \sum \Vert \chi_j f \Vert_{DV^2_\Delta[I]}^2 
 \]
The rest of the argument is similar to the one 
for the second part of \eqref{ixsum}.
\end{proof}

We conclude this section with the proof of Proposition~\ref{plin}.
For $t \geq 0$ we consider  the solution $u$ to the inhomogeneous equation
\[ 
i \partial_t u + \Delta u = f \qquad u(0)= u_0 
\]
 We set $u(t)=0$ for $t <0$. Then from the definitions  we immediately obtain 
the linear bound
\begin{equation} 
\label{RHS}
 \Vert u \Vert_{U^p} \le \Vert u_0 \Vert_{L^2} + \Vert f \Vert_{DU^p}. 
\end{equation} 

We now consider the bound \eqref{en-lin}. The frequency localization
commutes with the Schr\"o\-din\-ger operator, and it suffices to verify
\eqref{en-lin} for a fixed dyadic frequency range $\lambda$. We can
also restrict our attention to a time interval $J=[a,b]$ with 
$|J| = \lambda^{4s}$. There we need to show that
\begin{equation}
\| u_\lambda \|_{X_\lambda[J]} \lesssim \|u_\lambda(a)\|_{L^2} + 
\|f_\lambda\|_{Y_\lambda[J]}, \qquad (i \partial_t -\Delta) u_\lambda = f_\lambda
\label{diaden}\end{equation}
which follows directly from \eqref{RHS} and the definitions of the norms.

  For the second estimate \eqref{le-lin} we again localize to a dyadic frequency
$\lambda$. Let us first consider $\lambda > \Lambda$; there it takes the form
\[
\sup_{j \in \Z} \sum_{J \subset I}^{|J|=\lambda^{4s}}
\| \chi_j^\lambda u_\lambda \|_{X_\lambda[J]}^2 \lesssim 
\sup_{j \in \Z} \sum_{J \subset I}^{|J|=\lambda^{4s}} 
\left(\lambda^{-4s} \|\chi_j^\lambda u_\lambda\|_{L^2[J]}^2 + 
\|\chi_j^\lambda f_\lambda\|_{Y_\lambda[J]}^2\right)
\]
This in turn follows after integration over $t \in J$, 
 $J$ summation and $k$ summation 
from the next estimate:
\[
 \| \chi_j^\lambda u_\lambda \|_{X_\lambda[J]} \lesssim 
\sum_{k \in \Z} \langle j-k \rangle^{-N} \left(
\|\chi_k^\lambda u_\lambda(t)\|_{L^2} + 
\|\chi_k^\lambda f_\lambda\|_{Y_\lambda[J]}\right)
\]
This is equivalent to considering an inhomogeneous Cauchy problem 
in an interval $J=[a,b]$ with $|J| = \lambda^{4s}$,
\[
 (i \partial_t -\Delta) u_\lambda = P_\lambda \chi_k^\lambda f, \qquad  
u_\lambda(a) = \chi_k^\lambda u_0^\lambda
\]
and proving that 
\[
\| \chi_j^\lambda u_\lambda\|_{X_\lambda[J]} \lesssim 
 \langle j-k \rangle^{-N} \left( \|u_0\|_{L^2} + \|f_\lambda\|_{Y_\lambda[J]}\right)
\]
For $j=k+O(1)$ this is a direct consequence of \eqref{diaden}. For $j$ away from $k$
this follows from favorable bounds on the kernel $K_{jk}$ 
of $\chi_j^\lambda e^{it\Delta} P_\lambda \chi_j$, which satisfies
the rapid decay bounds
\[
|\partial_x^\alpha \partial_y^\beta \partial_t^\gamma K_{j,k}(t,x,y)| 
\leq c_{\alpha \beta \gamma}  \lambda^{-N}  \langle j-k \rangle^{-N},
\qquad |t| \leq \delta t_\lambda, \quad |j-k| \gg 1
\]
If $\lambda = \Lambda$ we apply the same argument to $\partial_x u_\Lambda$. 

\section{Linear and bilinear estimates}
\label{sbilin}

Solutions to the homogeneous equation,
\begin{equation}
iv_t -\Delta v = 0, \qquad v(0) = v_0
\label{hom}\end{equation}
 satisfy the Strichartz estimates:

\begin{proposition}
Let $p,q$ be indices satisfying 
\begin{equation}
\frac{2}p + \frac{1}q = \frac12, \qquad 4 \leq p \leq \infty
\label{pq}\end{equation}
Then the solution $u$ to \eqref{hom} satisfies 
\[
\|v\|_{L^p_t L^q_x} \lesssim \|v_0\|_{L^2}
\]
\end{proposition}
In particular we note the pairs of indices $(\infty,2)$, $(6,6)$ and
$(4,\infty)$.  As a straightforward consequence we have

\begin{cor}
 Let $p,q$ be indices satisfying  \eqref{pq}. 
Then 
\[
\|v\|_{L^p_t L^q_x} \lesssim \|v\|_{U^p}.
\]
\label{upstr}\end{cor}
The proof is straightforward, since it suffices to do it for atoms. 
By duality we also obtain

\begin{cor}
 Let $p,q$ be indices satisfying  \eqref{pq}.
Then 
\[
 \|v\|_{DV^{p'}}  \lesssim \|v\|_{L^{p'}_t L^{q'}_x} 
\]
\label{lpdual}\end{cor}

The second type of estimates we use are bilinear:

\begin{proposition}\label{bilinear} 
 Let $\lambda > 0$. Assume that $u,v$ are solutions to 
the homogeneous Schr\"odinger equation \eqref{hom}.
Then 
\begin{equation}
  \| P_{> \lambda}( u v)\|_{L^2} \lesssim \lambda^{-\frac12} \|u_0\|_{L^2} \|v_0\|_{L^2}
\end{equation}

\end{proposition}
\begin{proof}
In the Fourier space we have 
\[
\hat u(\tau,\xi) = \hat u_0(\xi) \delta_{\tau - \xi^2}, \qquad
\hat v(\tau,\xi) = \hat v_0(\xi) \delta_{\tau - \xi^2}
\]
Then 
\[
\widehat{uv}(\tau,\xi) = \int_{\xi_1+\xi_2 = \xi}  
 \hat u_0(\xi_1) \hat v_0(\xi_2) \delta_{\tau - \xi_1^2-\xi_2^2} d \xi_1
\]
which gives
\[
\widehat{uv}(\tau,\xi) =  \frac1{2|\xi_1-\xi_2|} (\hat u_0(\xi_1) \hat v_0(\xi_2) 
+ \hat  u_0(\xi_2) \hat v_0(\xi_1))
\]
where $\xi_1$ and $\xi_2$ are the solutions to
\[
\xi_1^2 +\xi_2^2 = \tau, \qquad \xi_1+\xi_2 = \xi
\]
We have 
\[
d\tau d\xi = 2|\xi_1-\xi_2| d\xi_1 d\xi_2
\]
therefore we obtain
\[
\|P_{>\lambda} (uv)\|_{L^2}^2 \le   \int_{|\xi_1-\xi_2| \gtrsim \lambda} |\hat u_0(\xi_1)|^2 |\hat v_0(\xi_2)|^2
|\xi_1-\xi_2|^{-1}d\xi_1 d\xi_2
\]
The conclusion follows.
\end{proof}

As a consequence we obtain

\begin{cor}
The following estimates hold:
\begin{equation}
   \Vert P_\lambda (u v) \Vert_{L^2} \lesssim
 \lambda^{-1/2} \Vert u \Vert_{U^2_\Delta} \Vert v \Vert_{U^2_\Delta}. 
\end{equation}
\begin{equation}
  \| u_\lambda  v_\mu\|_{L^2} \lesssim \mu^{-\frac12} 
  \|u_\lambda \|_{U^2_\Delta} \|v_\mu\|_{U^2_\Delta}, \qquad \lambda \ll \mu
\label{biuu}\end{equation}

\end{cor}
Again it suffices to prove these estimates for atoms, and then 
for solutions to the homogeneous Schr\"oder equation. But this follows from
  the bilinear estimate of Proposition \ref{bilinear}.

\section{The cubic nonlinearity}
\label{snonlin}
In this section we prove Proposition~\ref{pnonlin}.  For a dyadic
frequency $\lambda$ we estimate the nonlinearity $|u|^2 u$ at
frequency $\lambda$ in a time interval $I$ of length $\lambda^{4s}$ in
$DU^2_\Delta$. By duality this leads to a study of a quadrilinear form
of the type
 \[
J=  \int \chi_I u_{\lambda_1} \bar u_{\lambda_2} 
u_{\lambda_3} \bar u_{\lambda_4} dx dt 
\]
The position of the complex conjugates is of little importance in the
sequel. We will assume that $\lambda_1 \le \lambda_2 \le \lambda_3 \le
\lambda_4$; some of the constants in the next lemma improve
if the complex conjugates are placed differently, but this plays 
no role in our subsequent analysis.

\begin{lemma} \label{est} Let $I$ be any compact interval. 
Then the following estimates hold: 

{\bf A).} If $\lambda_1 \sim \lambda_2 \sim \lambda_3 \sim \lambda_4 $ then
\begin{equation} \label{ineq:A2} 
 |J| \lesssim  \|u_{\lambda_1}\|_{L^2}\|u_{\lambda_2}\|_{V^2_\Delta}
\|u_{\lambda_3}\|_{V^2_\Delta}\|u_{\lambda_4}\|_{V^2_\Delta}
\end{equation}
\begin{equation} \label{ineq:A} 
 |J| \lesssim  |I|^{\frac12}  \|u_{\lambda_1}\|_{V^2_\Delta}
\|u_{\lambda_2}\|_{V^2_\Delta}
\|u_{\lambda_3}\|_{V^2_\Delta}\|u_{\lambda_4}\|_{V^2_\Delta}
\end{equation}

{\bf B).}  If $\lambda_1 \sim \lambda_2 \ll  \lambda_3 \sim \lambda_4$ then
\begin{equation} \label{ineq:B2}
 |J| \lesssim \lambda_1^{\frac12}
\lambda_4^{-\frac12}\|u_{\lambda_1}\|_{L^2}\|u_{\lambda_2}\|_{U^2_\Delta}
\|u_{\lambda_3}\|_{U^2_\Delta}\|u_{\lambda_4}\|_{U^2_\Delta}
\end{equation} 
\begin{equation} \label{ineq:B*2}
 |J| \lesssim  \lambda_1^{\frac12}
\lambda_4^{-\frac12} \|u_{\lambda_1}\|_{U^2_\Delta}\|u_{\lambda_2}\|_{U^2_\Delta}
\|u_{\lambda_3}\|_{L^2}\|u_{\lambda_4}\|_{U^2_\Delta} 
\end{equation} 
\begin{equation} \label{ineq:B}
 |J| \lesssim \lambda_4^{-1}\|u_{\lambda_1}\|_{U^2_\Delta}\|u_{\lambda_2}\|_{U^2_\Delta}
\|u_{\lambda_3}\|_{U^2_\Delta}\|u_{\lambda_4}\|_{U^2_\Delta}
\end{equation} 

{\bf C).}  If $ \lambda_1 \ll \lambda_2 \ll \lambda_3 \sim \lambda_4$ 
\begin{equation} \label{ineq:C2}
 |J| \lesssim \lambda_1^{\frac12} \lambda_4^{-\frac12}
\|u_{\lambda_1}\|_{L^2}\|u_{\lambda_2}\|_{U^2_\Delta}
\|u_{\lambda_3}\|_{U^2_\Delta}\|u_{\lambda_4}\|_{U^2_\Delta}
\end{equation} 
\begin{equation} \label{ineq:C*2}
 |J| \lesssim 
\min\{ \lambda_1^{\frac12} \lambda_2^{-1}
\lambda_4^{1/2},1 \} \lambda_2^{\frac12} \lambda_4^{-\frac12} 
\|u_{\lambda_1}\|_{U^2_\Delta}\|u_{\lambda_2}\|_{L^2}
\|u_{\lambda_3}\|_{U^2_\Delta}\|u_{\lambda_4}\|_{U^2_\Delta}
\end{equation} 
\begin{equation} \label{ineq:C**2}
 |J| \lesssim \lambda_1^{\frac12} \lambda_4^{-\frac12}
\|u_{\lambda_1}\|_{U^2_\Delta}\|u_{\lambda_2}\|_{U^2_\Delta}
\|u_{\lambda_3}\|_{L^2}\|u_{\lambda_4}\|_{U^2_\Delta}
\end{equation} 
\begin{equation} \label{ineq:C}
 |J| \lesssim 
\min\{ \lambda_1^{\frac12} \lambda_2^{-1}
\lambda_4^{1/2},1 \} \lambda_4^{-1}\|u_{\lambda_1}\|_{U^2_\Delta}\|u_{\lambda_2}\|_{U^2_\Delta}
\|u_{\lambda_3}\|_{U^2_\Delta}\|u_{\lambda_4}\|_{U^2_\Delta}
\end{equation} 

{\bf D).} If $ \lambda_1 \ll \lambda_2 \sim \lambda_3 \sim \lambda_4$ then
\begin{equation} \label{ineq:D2}
 |J| \lesssim 
 \lambda_1^{\frac12}   \lambda_4^{-\frac12} 
 \|u_{\lambda_1}\|_{L^2}\|u_{\lambda_2}\|_{U^2_\Delta}\|u_{\lambda_3}\|_{U^2_\Delta}\|u_{\lambda_4}\|_{U^2_\Delta}
\end{equation} 
\begin{equation} \label{ineq:D*2}
 |J| \lesssim 
 \lambda_1^{\frac12}   \lambda_4^{-\frac12} 
 \|u_{\lambda_1}\|_{U^2_\Delta}\|u_{\lambda_2}\|_{L^2}
\|u_{\lambda_3}\|_{U^2_\Delta}\|u_{\lambda_4}\|_{U^2_\Delta}
\end{equation} 
\begin{equation} \label{ineq:D**2}
 |J| \lesssim 
 \lambda_1^{\frac14}   \lambda_4^{-\frac14} 
 \|u_{\lambda_1}\|_{U^2_\Delta}\|u_{\lambda_2}\|_{U^2_\Delta}
\|u_{\lambda_3}\|_{L^2}\|u_{\lambda_4}\|_{U^2_\Delta}
\end{equation} 
\begin{equation} \label{ineq:D}
 |J| \lesssim 
 \lambda_1^{\frac14}   \lambda_4^{-\frac14} 
 \|u_{\lambda_1}\|_{U^2_\Delta}\|u_{\lambda_2}\|_{U^2_\Delta}
\|u_{\lambda_3}\|_{U^2_\Delta}\|u_{\lambda_4}\|_{U^2_\Delta}
\end{equation}

 \end{lemma} 
 It is worth noting that the length of the interval enters only in
 \eqref{ineq:A}. In all other cases the cutoff $\chi_I$ can be 
safely discarded. The bounds in parts B and C improve if the 
none or both of the high frequency factors have complex conjugates.
We also note the weaker bound \eqref{ineq:D**2} when the 
complex conjugates fall on the first and third factor; this
directly leads to the weaker bound in \eqref{ineq:D}, and causes 
some small difficulties later on.

We also remark that combining the results in 
\eqref{ineq:A2}, \eqref{ineq:B2}, \eqref{ineq:B*2},\eqref{ineq:C2}, \eqref{ineq:C*2}
\eqref{ineq:D2}, and \eqref{ineq:D*2} we obtain by duality
\begin{cor}
Suppose that $\lambda_1 \leq \lambda_2 \leq \lambda_3$ and $\lambda_0 \lesssim
\lambda_2$. Then
\begin{equation} \label{tril2}
\| P_{\lambda_0} (u_{\lambda_1} \bar u_{\lambda_2}u_{\lambda_3})\|_{L^2}
\lesssim \lambda_0^\frac12 \lambda_3^{-\frac12} 
 \|u_{\lambda_1}\|_{U^2_\Delta}\|u_{\lambda_2}\|_{U^2_\Delta}
\|u_{\lambda_3}\|_{U^2_\Delta}
\end{equation}
\end{cor}

\begin{proof} 
{\bf A.} For \eqref{ineq:A2} we use the $L^6$ Strichartz estimate.
For \eqref{ineq:A2} we use the $L^8L^4$
Strichartz estimate and H\"older's inequality to obtain
\[  
J \lesssim  |I|^{\frac12} \prod_{j=1}^4 \Vert \chi_I u_{\lambda_j}  
\Vert_{L^8 L^4} 
\lesssim  |I|^{\frac12} \prod_{j=1}^4 \Vert \chi_I u_{\lambda_j}  \Vert_{U^8} 
\lesssim  |I|^{\frac12} \prod_{j=1}^4 \Vert \chi_I u_{\lambda_j}  \Vert_{V^2} 
\]

{\bf B.} For both \eqref{ineq:B2} and \eqref{ineq:B*2} we use an $L^2$ bilinear
estimate for $u_{\lambda_2} u_{\lambda_4}$, and the remaining two
factors are estimated in $L^2$ respectively $L^\infty L^2$ with an
added Bernstein inequality for $u_{\lambda_1}$.
For \eqref{ineq:B} we use two $L^2$ bilinear
estimates for $u_{\lambda_1} u_{\lambda_3}$, respectively 
$u_{\lambda_2} u_{\lambda_4}$.

{\bf C.} The bounds \eqref{ineq:C2} and \eqref{ineq:C**2} are similar
to \eqref{ineq:B2} and \eqref{ineq:B*2}. The same argument also yields
the $ \lambda_2^{\frac12} \lambda_4^{-\frac12} $ factor in
\eqref{ineq:C*2}.  To complete the proof of \eqref{ineq:C*2} we
observe that we can harmlessly insert a projector $\tilde
P_{\lambda_2} (u_{\lambda_3} \bar u_{\lambda_4})$ in the last
product. Here and later, $\tilde P_{\lambda}$ denotes a wider
frequency $\lambda$ projector, for instance $\tilde P_{\lambda}=
\sum_{\mu \sim \lambda} P_\mu$.  By the $L^2$ bilinear estimate we
have
\[
\| \tilde P_{\lambda_2} (u_{\lambda_3} \bar u_{\lambda_4})\|_{L^2}
\lesssim \lambda_2^{-\frac12} \|u_{\lambda_3}\|_{U^2_\Delta}
\|u_{\lambda_4}\|_{U^2_\Delta} 
\]
and using the $L^\infty$ bound for $u_{\lambda_1}$ we obtain the second desired 
factor $\lambda_1^{\frac12} \lambda_2^{-\frac12}$ in  \eqref{ineq:C*2}.

To prove \eqref{ineq:C} we decompose each factor into a term with
modulation $\gtrsim \lambda_2 \lambda_4$, and one with smaller
modulation,
\[
 u_{\lambda_j}= u_{\lambda_j}^h + u_{\lambda_j}^l
 \]
 If all four modulations are low, then a simple frequency-modulation
 analysis shows that $J=0$. Hence we assume without any restriction in
 generality that one of the factors is at high modulation. For that factor 
we have an $L^2$ bound, see \eqref{embed} and \eqref{XSB}, 
\[
\| u_{\lambda_j}^h\|_{L^2} \lesssim (\lambda_2 \lambda_4)^{-\frac12}
\|  u_{\lambda_j}\|_{U^2_\Delta} 
\]
Hence the constant in 
\eqref{ineq:C} is obtained by adding a  $(\lambda_2 \lambda_4)^{-\frac12}$
factor to each of the constants in \eqref{ineq:C2}-\eqref{ineq:C**2}
and summing them up.

{\bf D.} Case D is similar to case C except for the inequality \eqref{ineq:D**2},
where we only obtain the worse factor $\lambda_1^{\frac14}
\lambda_4^{-\frac14}$. The difference there is that $\xi_2$ and
$\xi_4$ no longer need to have dyadic separation so we cannot use
directly the bilinear $L^2$ bound.  To address this issue we split the
problem into two cases by writing 
\[
u_{\lambda_2} u_{\lambda_4} = P_{> \lambda} (u_{\lambda_2} u_{\lambda_4})
+P_{< \lambda} (u_{\lambda_2} u_{\lambda_4}),
\qquad \lambda =  (\lambda_1 \lambda_4)^{\frac12}
\]
For the first term we use the bilinear $L^2$ bound to obtain
\[
\| P_{> \lambda} (u_{\lambda_2} u_{\lambda_4})\|_{L^2}
\lesssim (\lambda_1 \lambda_4)^{-\frac14} \|u_{\lambda_2}\|_{U^2_\Delta}
\| u_{\lambda_4}\|_{U^2_\Delta}
\]
and conclude with the pointwise bound for $u_{\lambda_1}$.

For the second term we have orthogonality with respect to 
frequency intervals of size $\lambda$, therefore the problem reduces 
to the case when $u_{\lambda_2}$ and $ u_{\lambda_4}$ are frequency localized 
in $\lambda$ sized intervals.
Then we use
the $L^2$ bilinear bound for $u_{\lambda_1} u_{\lambda_4}$ gaining a 
$ \lambda_4^{-\frac12}$ factor, and then use Bernstein for $u_{\lambda_2}$
(now localized on the $(\lambda_1 \lambda_4)^{\frac12}$ scale)
for a loss of  $(\lambda_1 \lambda_4)^{\frac14}$.
\end{proof}

We continue with the proof of Proposition~\ref{pnonlin}.

\begin{proof}[Proof of Proposition~\ref{pnonlin}.]

For the $Y^s_\Lambda$ bound we need to estimate the trilinear expression 
\[
P_\lambda (u_{\lambda_1} \bar u_{\lambda_2} u_{\lambda_3})
\]
in $Y_\lambda[I]$ over intervals of size $|I| = \lambda^{4s}$, and
then square sum with respect to all frequencies $\lambda$,
$\lambda_1$, $\lambda_2$, $\lambda_3$. In the case of the
$Y^s_{\Lambda,le}$ bound we have to estimate 
the better, localized, trilinear expression 
\[
\chi_j^\lambda P_\lambda (u_{\lambda_1} \bar u_{\lambda_2} u_{\lambda_3})
\]
in $Y_\lambda[I]$, but we need to perform an additional summation with
respect to the time interval $I$.

We separately consider several cases depending on the 
relative size of all $\lambda$'s. In order for the output to be nonzero
we must be in one of the following two cases:

\smallskip
1)  $\max \{\lambda_{1}, \lambda_2,\lambda_3\} \sim \lambda$.
\smallskip

2) $\{\lambda_{1}, \lambda_2,\lambda_3\}  = \{ \alpha,\mu,\mu\}$
with $\lambda \ll \mu$, $\alpha \leq \mu$.

\smallskip 
Here we allow for a slight abuse of notation, as the two highest
$\lambda_j$'s need not be equal but merely comparable.  We will
consider these two cases separately. In the second case we will
subdivide into further cases depending on the relative size of
$\alpha$ and $\lambda$.

The space $Y_\lambda[I]$ is a weighted sum of an $L^2$ space and an
$DU^2_\Delta$ space. In Case 1 we will estimate the cubic term only in $L^2$.
In Case 2 we will estimate the cubic term in both $L^2$ and
$DU^2_\Delta$. The estimates in either spaces are good enough to complete
the argument, and hence there is some redundancy. Nevertheless we find
it instructive to do the extra work, as it shows that that this
argument does not break at $s = -1/4$.

We remark, though, that in order to continue it below $s = -1/4$ some
extra care is required as the balance of the spatial scales
changes. We also remark that this difficulty disappears exactly at $s
= -1/4$, when all spatial scales coincide.

{\bf Case 1:} $\max \{\lambda_{1}, \lambda_2,\lambda_3\} \sim
\lambda$.  This case imposes no restrictions on $s$ beyond $ s\ge
-1/2$.  Instead it makes the arguments for the length of the time
intervals in our $X^s_\Lambda$, $X^s_{\Lambda,le}$, $Y^s_\Lambda$ and
$Y^s_{\Lambda,le}$ norms precise. In this case we restrict ourselves
to the $L^2$ bound. Using duality and \eqref{ineq:A2} we obtain
\[
\begin{split} 
\|  \chi_I  u_{\lambda_1} u_{\lambda_2} u_{\lambda_3} \|_{L^2}
\lesssim & \ \| u_{\lambda_1}\|_{V^2[I]}  \|u_{\lambda_2}\|_{V^2[I]}
\| u_{\lambda_3} \|_{V^2[I]} 
\\ \lesssim & \ (\lambda_{1} \lambda_{2} \lambda_3)^{-s}
 \| u_{\lambda_1}\|_{X^s_\Lambda}  \|u_{\lambda_2}\|_{X^s_\Lambda}
\| u_{\lambda_3} \|_{X^s_\Lambda} 
\end{split}
\]
For intervals $I$ of size 
$|I| = \lambda^{4s}$ this gives
\[
\|P_\lambda  (u_{\lambda_1} u_{\lambda_2} u_{\lambda_3}) \|_{Y^s_\Lambda[I]}
\lesssim  (\lambda_{1} \lambda_{2} \lambda_3)^{-s} \lambda^{3s}
 \| u_{\lambda_1}\|_{X^s_\Lambda}  \|u_{\lambda_2}\|_{X^s_\Lambda}
\| u_{\lambda_3} \|_{X^s_\Lambda} 
\]
where the $\lambda_{1,2,3}$ summations are straightforward. 

For the estimate in  $Y^s_{\Lambda,le}$ we observe that 
\[
\| \chi_j^\lambda\chi_I  P_\lambda  ( u_{\lambda_1} u_{\lambda_2} u_{\lambda_3}) \|_{L^2}
\lesssim  \sum_l \langle j-l \rangle^{-N} 
  \| \chi_l^\lambda  \chi_I  u_{\lambda_1} u_{\lambda_2} u_{\lambda_3} \|_{L^2}.
\]
If $\lambda \gg \Lambda$ then the same argument as above applies since
the square summability with respect to time intervals is inherited
from $u_{\lambda_{max}}$. If $\lambda \sim \Lambda$ then we apply the
argument to $\partial _x( u_{\lambda_1} u_{\lambda_2} u_{\lambda_3})$;
then the square summability with respect to $I$ is inherited from the
differentiated factor.

{\bf Case 2:} $\{\lambda_{1}, \lambda_2,\lambda_3\}  = \{ \alpha,\mu,\mu\}$
with $\lambda \ll \mu$. We subdivide this as follows:

{\bf Case 2(a):} $ \lambda \sim \alpha \ll \mu$.  This case and the
next is where we gain most from the local energy decay bounds.  This
case requires no explicit restriction on $s$ beyond $s \ge -1/2$; instead it
determines the power of $\lambda$ in the $DU^2_\Delta$ component of $Y^s$ and
$Y^s_{le}$. We remark however that for the argument below it is important 
to know that the frequency $\mu$ spatial scale is larger than the frequency 
$\lambda$ spatial scale; this breaks down for $s < -1/4$ therefore
the above mentioned power of $\lambda$ would have to be adjusted
for such $s$.

The placement of the complex conjugates is irrelevant here, therefore
we let, say, $\lambda_1= \alpha$. We decompose each of the factors as
\[
u_{\lambda_i} = \sum_{j \in \Z} u_{\lambda_i,j},  \qquad
u_{\lambda_i,j} = \tilde P_{\lambda_i} (\chi_j^{\lambda_i} u_{\lambda_i})
\]
preserving the frequency localizations. The spatial localizations
are not preserved but the tails are negligible,
\begin{equation}\label{tails}
  |\chi_k^{\lambda_i} u_{\lambda_i,j}| \lesssim |k-j|^{-N} \lambda_i^{-N} 
\| \chi_j^{\lambda_i} u_{\lambda_i}\|_{L^\infty L^2}, \qquad |k-j| \gg 1
\end{equation}
 For $j \in \Z$ and $|I|=
\lambda^{4s}$ we define the localized trilinear expressions
\[
f_{I,j}:=\chi_I  u_{\lambda_1,j}
\bar u_{\lambda_2} u_{\lambda_3} 
 =     \chi_I  u_{\lambda_1,j}  \sum_{I' \subset I}^{|I'| = \mu^{4s}} \sum_{k_2,k_3}
  \chi_{I'}   \bar u_{\lambda_2,k_2} \cdot 
 \chi_{I'}  u_{\lambda_3,k_3} 
\]
Due to the above mentioned ordering of spatial scales there is some unique 
(up to $O(1)$) $k = k(j)$ so that the supports of $\chi_j^{\lambda_1}$ and
$\chi_k^{\lambda_{2,3}}$ overlap.
We will first bound $P_\lambda f_{I,j}$ in  $L^2$.  Using
\eqref{ineq:B2}  and duality  we obtain
\[
\|P_\lambda f_{I,j} \|_{L^2} \lesssim 
\frac{\lambda^\frac12}{\mu^{\frac12}} \|u_{\lambda_1,j}\|_{U^2_\Delta[I]}
\sum_{k_2,k_3}   \sum_{I' \subset I}^{|I'| = \mu^{4s}}  
 \frac{\| u_{\lambda_2,k_2}\|_{U^2_\Delta[I']}}{\langle k_2 - k(j)\rangle^{N}} 
\frac{\| u_{\lambda_3,k_3}\|_{U^2_\Delta[I']}}{\langle k_3 - k(j)\rangle^{N}}
\]
where the rapid decay away from when $k_{2,3} = k(j)$ is due to
\eqref{tails}.  Next we use Cauchy-Schwartz with respect to $I'$ and
then sum with respect to $k_2,k_3$ to get
\begin{equation}
\|P_\lambda f_{I,j} \|_{L^2} \lesssim 
\lambda^\frac12 \mu^{-\frac12-2s} \|\chi^\lambda_j u_{\lambda_1}\|_{U^2_\Delta[I]}
  \| u_{\lambda_2}\|_{X^s_{\Lambda,le}}
\| u_{\lambda_3}\|_{X^s_{\Lambda,le}}
\label{quad-app2}\end{equation}
The square summability 
with respect to space or time intervals is inherited from $u_{\lambda_{1}}$,
so we conclude that
\begin{equation}
\| P_\lambda (u_{\lambda_1} \bar u_{\lambda_2} u_{\lambda_3})\|_{Y^s_\Lambda}
\lesssim \lambda^{\frac12+2s} \mu^{-\frac12-2s} 
\|u_{\lambda_1}\|_{X^s_\Lambda}
  \| u_{\lambda_2}\|_{X^s_{\Lambda,le}}
\| u_{\lambda_3}\|_{X^s_{\Lambda,le}}
\label{xs-case22}\end{equation}
and similarly for the $Y^s_{\Lambda,le}$ norm. The summation with
respect to $\lambda$ and $\mu$ is straightforward.

We remark that this approach gives a better bound for high modulations
but works only when $s \geq -\frac14$. However, the estimate in $DU^2$  
 shows that there is some room beyond $s = -1/4$.

>From \eqref{ineq:B} by duality and  \eqref{tails} for the tails we obtain
\[
\|P_\lambda f_{I,j} \|_{DV^2_\Delta[I]} \lesssim 
\mu^{-1} \|\chi^\lambda_j u_{\lambda_1}\|_{U^2_\Delta[I]}
  \sum_{k_2,k_3} \sum_{I' \subset I}^{|I'| = \mu^{4s}}  
 \frac{\| u_{\lambda_2,k_2}\|_{U^2_\Delta[I']}}{\langle k_2 - k(j)\rangle^{N}} 
\frac{\| u_{\lambda_3,k_3}\|_{U^2_\Delta[I']}}{\langle k_3 - k(j)\rangle^{N}}.
\]
After  Cauchy-Schwartz with respect to $I'$ and $k_2$, $k_3$ summation we  obtain
\begin{equation}
\|P_\lambda f_{I,j} \|_{DV^2_\Delta[I]} \lesssim 
\mu^{-1-2s} \|\chi^\lambda_j u_{\lambda_1}\|_{U^2_\Delta[I]}
  \| u_{\lambda_2}\|_{X^s_{\Lambda,le}}
\| u_{\lambda_3}\|_{X^s_{\Lambda,le}}   
\label{quad-app}\end{equation}
Comparing this with \eqref{quad-app2} we see that \eqref{quad-app2}
is stronger at modulations $\geq \lambda \mu$ while \eqref{quad-app}
is stronger  at modulations $\leq \lambda \mu$. Precisely, from 
\eqref{quad-app2} we obtain
\[
\begin{split}
\|Q_{> \lambda \mu} P_\lambda f_{I,j}\|_{Y_\lambda[I]}
\lesssim & \ \lambda^{1+2s}
\|Q_{> \lambda \mu} P_\lambda f_{I,j}\|_{DU^2_\Delta}
 \\ \lesssim & \ \lambda^{1/2+2s} \mu^{-1/2}  \| P_\lambda f_{I,j}\|_{L^2}
\\ \lesssim & \
  \lambda^{1+2s} \mu^{-1-2s} 
  \|\chi^\lambda_j u_{\lambda_1}\|_{U^2_\Delta[I]}
  \| u_{\lambda_2}\|_{X^s_{\Lambda,le}}
\| u_{\lambda_3}\|_{X^s_{\Lambda,le}}
\\ \lesssim & \
\lambda^{1+3s} \mu^{-1-2s} \|\chi^\lambda_j u_{\lambda_1}\|_{X^s_\Lambda[I]}
  \| u_{\lambda_2}\|_{X^s_{\Lambda,le}}
\| u_{\lambda_3}\|_{X^s_{\Lambda,le}}
\end{split} 
\]
while from \eqref{quad-app} we have
\[
\|Q_{< \lambda \mu} P_\lambda f_{I,j}\|_{Y_\lambda}
\lesssim  \log(\mu/\lambda) \lambda^{1+3s} \mu^{-1-2s} 
\|\chi^\lambda_j u_{\lambda_1}\|_{X^s_\Lambda[I]}
  \| u_{\lambda_2}\|_{X^s_{\Lambda,le}}
\| u_{\lambda_3}\|_{X^s_{\Lambda,le}}
\]
where the logarithmic loss is due to the number of dyadic regions
between modulations $\lambda^2$ and $\lambda \mu$ arising 
in the conversion of the $DV^2$ norm into a $DU^2$ norm, see \eqref{besov}.

Again the square summability with respect to space\footnote{Here, as
  well as in all the other cases, we want to use the second bound in
  \eqref{dixsum} rather than \eqref{ixsum}, so the spatial summation
  must precede the $DV^2$ to $DU^2$ conversion.} or time intervals is
inherited from $u_{\lambda_{1}}$, so we obtain an improved form of
\eqref{xs-case22}, namely
\begin{equation}
\!\! \| P_\lambda (u_{\lambda_1} \bar u_{\lambda_2} u_{\lambda_3})\|_{Y^s_\Lambda}
\lesssim \!
 \Big(\frac{\lambda}{\mu}\Big)^{1+2s}\! \log\!\left(\frac{\mu}{\lambda}\right)\!
\|u_{\lambda_1}\|_{X^s_\Lambda}
  \| u_{\lambda_2}\|_{X^s_{\Lambda,le}}
\| u_{\lambda_3}\|_{X^s_{\Lambda,le}}
\label{xs-case2}\end{equation}
and similarly for the $Y^s_{\Lambda,le}$ norm. Thus $s \geq -\frac14$ is 
more than enough.

{\bf Case 2(b):} $ \alpha \ll \lambda \ll \mu$. 
This case is similar to the previous case in that the 
local energy norms give a crucial gain in the estimates.
This case is also different from the previous case in a fundamental
way, namely that the interaction is nonresonant.  Precisely, either
the output or at least one of the inputs must have modulation at least
$\lambda \mu$. In the latter case, there is a further gain due to our
definition of the $X^s$ respectively $X^s_{le}$ spaces.
Unfortunately, this gain disappears as $\alpha$ gets small, so we
cannot take good advantage of it, and instead we end up repeating the
arguments of Case 2(a). 

Again the placement of the complex conjugates is irrelevant, so we let
$\lambda_1 = \alpha$.  However we readjust the definition of $f_{I,j}$
to
\[
f_{I,j} = \chi_I \tilde P_{\lambda_1} (\chi_j^\lambda u_{\lambda_1}) u_{\lambda_2}
u_{\lambda_3}
\]
using the larger spatial scale $\delta x_\lambda$ instead of 
$\delta x_\alpha$ for the cutoffs.
Using a dual form of \eqref{ineq:C*2} as well as \eqref{tails} for
off-diagonal tails we obtain the trilinear $L^2$ bound
\[\begin{split} 
\|P_\lambda f_{I,j} \|_{L^2} \lesssim  
\min\{\alpha^\frac12 \lambda^{-1} \mu^\frac12, 1\}
 \lambda^\frac12 \mu^{-\frac12} \|\chi^\lambda_j u_{\lambda_1}\|_{U^2_\Delta[I]}
    \sum_{k_2,k_3}  \sum_{I' \subset I}^{|I'| = \mu^{4s}} 
 \frac{\| u_{\lambda_2,k_2}\|_{U^2_\Delta[I']}}{\langle k_2 - k(j)\rangle^{N}} 
\frac{\| u_{\lambda_3,k_3}\|_{U^2_\Delta[I']}}{\langle k_3 - k(j)\rangle^{N}}.
\end{split} 
\]
On the other hand from \eqref{ineq:C} by duality  we obtain
\[
\begin{split}
\|P_\lambda f_{I,j} \|_{DV^2[I]} \lesssim 
\min\{\alpha^\frac12 \lambda^{-1} \mu^\frac12, 1\}
\mu^{-1} \|\chi^\lambda_j u_{\lambda_1}\|_{U^2_\Delta [I]}
   \sum_{k_2,k_3}  \sum_{I' \subset I}^{|I'| = \mu^{4s}}  
 \frac{\| u_{\lambda_2,k_2}\|_{U^2_\Delta[I']}}{\langle k_2 - k(j)\rangle^{N}} 
\frac{\| u_{\lambda_3,k_3}\|_{U^2_\Delta[I']}}{\langle k_3 - k(j)\rangle^{N}}.
\end{split} 
\]
Using the former for modulations $\geq \lambda \mu$ and the latter for
smaller modulations we obtain
\begin{equation} \label{2binter}
\begin{split} 
\|P_\lambda f_{I,j} \|_{Y_\lambda[I]}  \lesssim &\ \log(\lambda/\mu)
\min\{\alpha^\frac12 \lambda^{-1} \mu^\frac12, 1\}
\lambda^{1+2s} \mu^{-1} 
\\ & \times \|\chi^\lambda_j u_{\lambda_1}\|_{U^2_\Delta[I]}
 \sum_{k_2,k_3} \sum_{I' \subset I}^{|I'| = \mu^{4s}} 
 \frac{\| u_{\lambda_2,k_2}\|_{U^2_\Delta[I']}}{\langle k_2 - k(j)\rangle^{N}} 
\frac{\| u_{\lambda_3,k_3}\|_{U^2_\Delta[I']}}{\langle k_3 - k(j)\rangle^{N}}.
\end{split} 
\end{equation}
 After Cauchy-Schwarz with respect to $I'$ this gives
\[
\begin{split} 
\|P_\lambda f_{I,j} \|_{Y_\lambda[I]} \lesssim  \log(\lambda/\mu)
\min\{\alpha^\frac12 \lambda^{-1} \mu^\frac12, 1\}
\lambda^{1+2s} \mu^{-1-2s} 
  \|\chi^\lambda_j u_{\lambda_1}\|_{U^2_\Delta[I]}
  \|u_{\lambda_2}\|_{X^s_{\Lambda,le}}
\| u_{\lambda_3}\|_{X^s_{\Lambda,le}}
\end{split} 
\]
The square summability with respect to spatial intervals is inherited
from $u_{\lambda_{1}}$, so we conclude that
\begin{equation}
\begin{split} 
\| P_\lambda (u_{\lambda_1} \bar u_{\lambda_2} u_{\lambda_3})\|_{Y^s_\Lambda}
\lesssim  \log(\lambda/\mu)
\min\{\alpha^\frac12 \lambda^{-1} \mu^\frac12, 1\}
\alpha^{-s} \lambda^{1+3s} \mu^{-1-2s}
 \|u_{\lambda_1}\|_{X^s_\Lambda}
  \| u_{\lambda_2}\|_{X^s_{\Lambda,le}}
\| u_{\lambda_3}\|_{X^s_{\Lambda,le}}
\end{split}
\label{xs-case32}\end{equation}
The summation with respect to $\alpha$ and $\mu$ is straightforward.

The local energy $Y^s_{\Lambda,le}$ estimate does not pose additional
difficulties. We first estimate the frequency $\alpha$ factor in
\eqref{2binter} by 
\[
\|\chi^\lambda_j u_{\lambda_1}\|_{U^2_\Delta[I]} 
\lesssim \alpha^{-s} \|u_{\lambda_1}\|_{X^s_\Lambda}
\]  
Then we sum over all $I \subset [0,1]$ and use Cauchy-Schwarz  
with respect to $I' \subset [0,1]$.
Thus the time interval summation is inherited from the highest
frequencies, and we obtain the same constants as in \eqref{xs-case32}.

{\bf Case 2(c):} $ \lambda \ll \alpha \ll \mu$.  In this
case, as $\alpha$ increases, the usefulness of the local energy norms
decreases.  The interaction is still nonresonant, i.e.  either the output
or at least one of the inputs must have modulation at least $\lambda
\mu$. However, this time we can fully exploit the gain coming from 
the better bounds for high modulation inputs. On the other hand if the output 
has high modulation then we get to use the better constant in \eqref{ineq:C2}
(compared to \eqref{ineq:C*2}).

Again the placement of complex conjugates does not matter,
so we let $\alpha = \lambda_1$ but we return to our original notation 
in Case 2(a),
\[
f_{I,j} :=   u_{\lambda_1,j} \bar u_{\lambda_2}  u_{\lambda_3} 
\]
We claim that the following bound holds for $|I| = \lambda^{4s}$:
\begin{equation}
 \|P_\lambda   f_{I,j}\|_{Y_\lambda[I]} 
\lesssim  C
\sup_{I' \subset I,\ |I'| = \alpha^{4s}} \|\chi^\alpha_j u_{\lambda_1}\|_{X_\alpha[I']}
\| u_{\lambda_2}\|_{X^s_{\Lambda,le}} \|u_{\lambda_3}\|_{X^s_{\Lambda,le}}
\label{c2c}\end{equation}
where $C=C(\lambda,\alpha,\mu)$ is given by
\[
C(\lambda,\alpha,\mu) =  \log(\mu/\lambda)  \lambda^{1+2s}\mu^{-1-2s}(
 \lambda^{\frac12} \alpha^{-\frac12} +  \alpha^{-1-2s}   
\min\{1,\lambda^\frac12 \alpha^{-1} \mu^{\frac12}\}).
\]
Here the second term is at most as large as the first if $s\ge -1/4$
and could be omitted.  Using \eqref{c2c} we conclude the proof in
this case. For the local energy decay norm $Y^s_{\Lambda,le}$ we square sum
over $I \subset [0,1]$ and inherit the square summability from
$u_{\lambda_1}$, so there is no further loss; we obtain
\[
\| P_\lambda(u_{\lambda_1} \bar u_{\lambda_2} u_{\lambda_3})\|_{Y^s_{\Lambda,le}}
\lesssim  \log(\mu/\lambda) \frac{\lambda^{\frac32+3s}} {\alpha^{\frac12+s}\mu^{1+2s}}
 \| u_{\lambda_1}\|_{X^s_{\Lambda,le}}
\| u_{\lambda_2}\|_{X^s_{\Lambda,le}} \|u_{\lambda_3}\|_{X^s_{\Lambda,le}}
\]
For the $Y^s_\Lambda$ norm we square sum over $j$ in \eqref{c2c}. 
The $j$ summation is  inherited from $v_{\lambda_1}$; however we cannot interchange
the $I'$ supremum with the square summation in $j$. Instead we relax 
the supremum to an $l^2$ norm, which allows us to interchange norms 
 but causes  
an $( \alpha/\lambda)^{-2s}$ loss  in the  $I'$ summation.
This yields the worse bound
 \[
\| P_\lambda(u_{\lambda_1} \bar u_{\lambda_2} u_{\lambda_3})\|_{Y^s_\Lambda}
\lesssim  \log(\mu/\lambda)
 \frac{\lambda^{\frac32+5s}} {\alpha^{\frac12+3s}\mu^{1+2s}} \| u_{\lambda_1}\|_{X^s_{\Lambda}}
\| u_{\lambda_2}\|_{X^s_{\Lambda,le}} \|u_{\lambda_3}\|_{X^s_{\Lambda,le}}
\]
It is easy  to check that  the $\alpha$ and $\mu$
summation  is favorable since $s > - 3/10$.

We now prove \eqref{c2c}. We begin by writing 
\[
f_{I,j} :=   \sum_{k_2,k_3} \chi_I u_{{\lambda_1},j}  
\bar u_{\lambda_2,k_2}  u_{\lambda_3,k_3} 
\]
The off-diagonal terms where $k_2$ or $k_3$ are away from $k(j)$ 
are estimated directly as in Case 2(a),(b) using the rapid decay
in \eqref{tails}. It remains to consider the diagonal contribution
which we write as 
\[
f^d_{I,j} :=   \chi_I u_{{\lambda_1},j}  
\bar u_{\lambda_2,k}  u_{\lambda_3,k}, \qquad k = k(j) 
\]
We actually obtain a finite sum of such terms, which we suppress in
the notation.

 We consider a time  interval decomposition of  $f^d_{I,j}$,
\[
 f^d_{I,j}
= 
\sum_{I' \subset I}^{|I'| = \alpha^{4s}} 
\chi_{I'} u_{{\lambda_1},j} \sum_{I'' \subset I'}^{|I''| = \mu^{4s}} 
\chi_{I''}   \bar u_{\lambda_2,k} \cdot\chi_{I''} 
  \bar u_{\lambda_3,k} := \sum_{I'' \subset I}^{|I''| = \mu^{4s}} f^d_{I'',j}
\]
Since we will use modulation truncations which are nonlocal in time,
for the rest of the argument we extend each of the three factors above
to solutions to the homogeneous Schr\"odinger equation outside $I'$
resp $I''$.  A key reason for working with these extensions is that
they satisfy better high modulation bounds than the original interval
localized functions, see \eqref{em2}.  Recalling that
 $E_J$ is the  extension
operator for the interval $J$ we define the extension
$f^e_{I'',j}$ of $f^d_{I'',j}$ by
\[
f^e_{I'',j} = E_{I'} (u_{{\lambda_1},j}) \overline{ E_{I''} (
 u_{\lambda_2,k})}  E_{I''} ( u_{\lambda_3,k}) 
\]
For $f^{e}_{I'',j}$ we establish a global $L^2$ bound, as well as a
stronger low modulation $DV^2_\Delta$ estimate; the balance between
these two bounds is at modulation $\alpha \mu$.  For the $L^2$ bound
we recall that \eqref{ineq:C2} holds on the whole real line. Hence we
obtain the $L^2$ estimate
\begin{equation}\label{c2cl2}
\|P_\lambda  f^e_{I'',j}\|_{L^2} \lesssim \ \lambda^\frac12 \mu^{-\frac12}
 \| u_{{\lambda_1},j}\|_{U^2[I']}
\| u_{\lambda_2,k}\|_{U^2[I'']}\| u_{\lambda_3,k}\|_{U^2[I'']} 
\end{equation}
Next we consider the low modulations $ Q_{\ll \alpha \mu} P_\lambda
f^e_{I'',j}$.  Since we are in a nonresonant case, this is nonzero
only if one of the three factors has high modulation ($\gtrsim \alpha
\mu$). But for the high modulations we have better bounds.   We write
\[
Q_{\ll \alpha \mu} P_\lambda f^e_{I'',j} = 
Q_{\ll \alpha \mu} P_\lambda (g^1_{I'',j}+ g^2_{I'',j}+ g^3_{I'',j}) 
\]
where
\[
\begin{split}
g^1_{I'',j} =  & \ Q_{\gtrsim \alpha \mu} E_{I'}
 u_{{\lambda_1},j} \cdot  \overline{ E_{I''}   u_{\lambda_2,k}}
\cdot E_{I''} u_{\lambda_3,k},
\\
g^2_{I'',j}= & \ Q_{\ll \alpha \mu} E_{I'}
 u_{{\lambda_1},j}  \cdot  \overline{Q_{\gtrsim \alpha \mu}   E_{I''}   u_{\lambda_2,k}} \cdot
E_{I''} u_{\lambda_3,k},
\\ 
g^3_{I'',j}= & \ Q_{\ll \alpha \mu} E_{I'} 
 u_{{\lambda_1},j}  \cdot \overline{ Q_{\ll \alpha \mu} E_{I''}   u_{\lambda_2,k}}
 \cdot Q_{\gtrsim \alpha \mu} E_{I''} u_{\lambda_3,k},
\end{split}
\]
We will only consider the first term $g^1_{I'',j}$; the analysis for the other
two terms is similar but the result is better since the high modulation gain
of $\alpha^{-1-2s}$ is replaced by $\mu^{-1-2s}$.

For the  high modulation truncation we use the $L^2$ bound in \eqref{em2}.
Then by  \eqref{ineq:C*2} and duality we estimate $P_\lambda g^{1}_{I'',j}$ as
\[
\begin{split} 
\| P_\lambda g^{1}_{I'',j}\|_{DV^2_\Delta} \lesssim & \ C_1
  \| Q_{\gtrsim \alpha \mu} E_{I'} u_{{\lambda_1},j} \|_{L^2} 
\|  u_{\lambda_2,k}\|_{U^2[I'']} 
\| u_{\lambda_3,k}\|_{U^2[I'']} 
\\ \lesssim & \  C_1 \alpha^{-1-2s} (\alpha \mu)^{-\frac12} \| u_{{\lambda_1},j}\|_{X_\lambda[I']}
\|  u_{\lambda_2,k}\|_{U^2[I'']} 
\| u_{\lambda_3,k}\|_{U^2[I'']} 
\end{split} 
\]
with $C_1 = \min\{\mu^{\frac12} \alpha^{-1} \lambda^{\frac12},1\}
\mu^{-\frac12} \alpha^{\frac12}$. Adding the similar bounds for 
$P_\lambda g^{2}_{I'',j}$ and $P_\lambda g^{3}_{I'',j}$ we obtain
\[
\|Q_{\ll \alpha \mu} P_\lambda f^e_{I'',j} \|_{DV^2_\Delta} \lesssim 
 C_1 \alpha^{-\frac32-2s} \mu^{-\frac12} \| u_{{\lambda_1},j}\|_{X_\lambda[I']}
\|  u_{\lambda_2,k}\|_{X_\lambda[I'']} 
\| u_{\lambda_3,k}\|_{X_\lambda[I'']} 
\]
We combine this with the high modulation control derived from 
\eqref{c2cl2} to conclude that
\begin{equation}\label{c2cdv2}
\|P_\lambda f^e_{I'',j} \|_{DV^2_\Delta} \lesssim 
 C_2 \| u_{{\lambda_1},j}\|_{X_\lambda[I']}
\|  u_{\lambda_2,k}\|_{X_\lambda[I'']} 
\| u_{\lambda_3,k}\|_{X_\lambda[I'']} 
\end{equation}
where 
\[
C_2 = C_1 \alpha^{-\frac32-2s} \mu^{-\frac12} + \lambda^{\frac12} \mu^{-1}
\alpha^{-\frac12}
\]
Since neither \eqref{c2cl2} nor \eqref{c2cdv2} 
contain modulation localizations, we can truncate both of them to the interval $I''$
and obtain the similar bounds for  $f^d_{I'',j}$.  Given the definition of the $Y_\lambda$
space, from \eqref{c2cl2} nor \eqref{c2cdv2} for $f^d_{I'',j}$ we obtain
\begin{equation}\label{c2cy}
\| f^d_{I'',j}\|_{Y_\lambda} \lesssim \log(\mu/\lambda) \lambda^{1+2s} C_2  
\| u_{{\lambda_1},j}\|_{X_\lambda[I']}
\|  u_{\lambda_2,k}\|_{X_\lambda[I'']} 
\| u_{\lambda_3,k}\|_{X_\lambda[I'']} 
\end{equation}
where the logarithmic factor counts the number of dyadic regions
between modulations $\lambda^2$ and $\alpha \mu$ arising in the
transition from $DV^2_\Delta$ to $DU^2_\Delta$.  Since $C=
\log(\mu/\lambda) \lambda^{1+2s} \mu^{-2s} C_2$, the bound \eqref{c2c}
follows from \eqref{c2cy} after summation over $I'' \subset I$.

{\bf Case 2(d):} $ \lambda \ll \alpha \sim  \mu $.

For the most part this case is identical to Case 2(c).
 The interaction is still nonresonant, i.e.  either the output
or at least one of the inputs must have modulation at least 
$\mu^2$.  The high modulation output bound is unchanged.
 We have singled out this case because of a peculiarity
which occurs when the middle (conjugated) factor 
is at high modulation (see \eqref{ineq:D**2}  and  \eqref{ineq:D}).
This leads to a constant $C_1 =  (\lambda/\mu)^\frac14$ in the 
bound for $g^{2}_{I'',j}$, which in turn yields a constant $C$ in 
the counterpart of \eqref{c2c} of the form
\[
C(\lambda,\mu,\mu) =   \log(\mu/\lambda)( \lambda^{\frac12} \mu^{-\frac32-2s}+ 
  \mu^{-1-2s}  \mu^{-1-2s}   ({\lambda}/{\mu})^{\frac14})
\]
The first part is as in Case 2(c), but for the second
we need to sum up with respect to $\mu$ in the expression 
\[
 \log(\mu/\lambda) \lambda^{1+3s} \mu^{-s}  ({\lambda}/{\mu})^{2s} 
  \mu^{-1-2s}  \mu^{-1-2s}    ({\lambda}/{\mu})^{\frac14}
= \log(\mu/\lambda)  \mu^{7s+\frac94} \lambda^{\frac54+5s} 
\]
which is favorable since $s > - 9/28$. We note that this is the worst among all cases 
we have considered. 
\end{proof}

\section{The energy conservation}

In this section we study the weighted energy conservation for
solutions $u$ to \eqref{nls}. In order to keep the notations and the
exposition as simple as possible, here and in the next section we will
restrict ourselves to the endpoint case $s = -\frac14$. This suffices
in order to obtain the $H^s$ energy estimates and to fully prove
Theorem~\ref{main}, but not the space-time estimates \eqref{mainxs} for
$s > -\frac14$. The arguments here can easily be adapted to all larger
$s$.
 The main result here is Proposition~\ref{p:en},
which will be used in the last section of the paper 
to prove the first part of Proposition~\ref{penergy}, namely 
the bound \eqref{en-nonlin}.

Given a positive multiplier $a$ we  set 
\[
E_0(u) = \langle A(D) u,u\rangle:=\|u\|_{H^a}^2
\]
For the straight $H^s_\Lambda$ energy conservation it suffices to take
\[
 a(\xi) = (\Lambda^2+\xi^2)^{s} 
\]
However, as in \cite{MR2353092}, in order to gain the uniformity in $t$
required by \eqref{xs} we need to allow a slightly larger class of
symbols.

\begin{definition}
a)  Let $\Lambda \ge 1$. Then  $S_\Lambda$ is the class of
  spherically symmetric symbols with the following properties:

(i) symbol regularity,
\[
| \partial^\alpha a(\xi)| \lesssim a(\xi) (\Lambda^2+\xi^2)^{-\alpha/2}
\]

(ii) decay at infinity,
\[
     a(\xi)\ge (\Lambda^2 + \xi^2)^{-1/2} 
\]
and 
\[ 
[0,\infty) \ni  \xi \to a(\xi) (\Lambda^2+\xi^2)^{1/2} 
\]
is nondecreasing. 

b) If $a$ satisfies (i) and (ii) then we say that $d$ is dominated by  $a$, 
$d \in S(a)$,  if 
 \[ |\partial^\alpha d| \lesssim a (\Lambda^2+\xi^2)^{-\alpha/2} \]
with constant depending only on $\alpha$. 
\end{definition}
For such symbols $a$ we denote by $X^a$ respectively $X^a_{le}$ the
spaces defined as $X^{-\frac14}_\Lambda$ respectively
$X^{-\frac14}_{\Lambda,le}$ but with the symbol
$(\Lambda^2+\lambda^{2})^{-\frac14} $ replaced by $a(\lambda)$. Here
the spatial and temporal scales are the ones corresponding to $s =
-\frac14$, namely $\delta x = 1$, $\delta t = \lambda^{-1}$.

We compute the derivative of $E_0$ along the flow,
\[
\frac{d}{dt} E_0(u) =  R_4(u) = 2 \Re \langle i A(D) u,|u|^2 u\rangle 
\]
We write $R_4$ as a multilinear operator in the Fourier space,
\[
R_4(u) = 2 \Re \int_{P_4} i a(\xi_0) \hat u(\xi_0)
\hat u(\xi_1) \overline{ \hat u(\xi_2)
\hat u(\xi_3)} d\sigma  
\]
where 
\[
P_4 = \{ \xi_0+\xi_1-\xi_2-\xi_3=0\}
\]
This can be symmetrized,
\[
R_4(u) = \frac12 \Re \int_{P_4} i
(a(\xi_0)+a(\xi_1)-a(\xi_2)-a(\xi_3)) \hat u(\xi_0) \hat u(\xi_1) 
\overline{\hat
u(\xi_2) \hat u(\xi_3)} d\sigma.
\]
Following a variation of the $I$-method, see Tao~\cite{MR2233925}-3.9
and references therein, we seek to cancel this term by
perturbing the energy, namely by
\[
E_1(u) =  \int_{P_4}
b_4(\xi_0,\xi_1,\xi_2,\xi_3) \hat u(\xi_0) \hat u(\xi_1) \overline{\hat
u(\xi_2) \hat u(\xi_3)} d\sigma
\]
To determine the best choice for $b_4$  we compute
\[
\begin{split}
\frac{d}{dt} E_1(u) =  \int_{P_4} i 
b_4(\xi_0,\xi_1,\xi_2,\xi_3)
(\xi_0^2+\xi_1^2- \xi_2^2- \xi_3^2) \hat u(\xi_0) \hat u(\xi_1) \overline{\hat
u(\xi_2) \hat u(\xi_3)} d\sigma  + R_6(u)
\end{split}
\]
where $R_6(u)$ is given by
\[
R_6(u) = 4 \Re \int _{\xi_0+\xi_1-\xi_2-\xi_3=0} i 
b_4(\xi_0,\xi_1,\xi_2,\xi_3) \widehat{|u|^2u}(\xi_0) \hat u(\xi_1) \overline{\hat
u(\xi_2) \hat u (\xi_3)} d\sigma
\]
To achieve the cancellation of the quadrilinear form we define $b_4$ 
by 
\begin{equation}
b_4(\xi_0,\xi_1,\xi_2,\xi_3) = -  \frac{a(\xi_0)+a(\xi_1)-a(\xi_2)-a(\xi_3)}
{\xi_0^2+\xi_1^2- \xi_2^2- \xi_3^2} \qquad \text{for }\
(\xi_0,\xi_1,\xi_2,\xi_3)
\in P_4.
\label{b4}\end{equation}
Summing up the result of our computation, we obtain
\begin{equation}\label{en-relation}
\frac{d}{dt} (E_0(u) +E_1(u)) = R_6(u)
\end{equation}
We integrate this relation to estimate $E_0(u)=\|u\|_{H^a}^2$ uniformly in time:

\begin{proposition}\label{p:en}
Let $a \in S_\Lambda$. Then for any $L^2$ solution $u$ to \eqref{nls} 
in the time interval $[0,1]$ we 
have
\begin{equation}
 \|u\|_{L^\infty H^a}^2 \lesssim \|u_0\|_{H^a}^2
+  \|u\|_{L^\infty H^a}^2 \|u\|_{L^\infty H^{-\frac12}_\Lambda}^2
+ \| u\|_{X^a \cap X^a_{le}}^2 
\| u\|_{X^{-\frac14}_\Lambda \cap X^{-\frac14}_{\Lambda,le}}^4 
\end{equation}
\end{proposition}

The proposition follows directly from the bounds for $E_1(u)$ and of $R_6(u)$
in Lemmas~\ref{e1bd}, \ref{r6} below. The aim of the rest of this section 
is to prove these two lemmas.

In order to estimate the size of $E_1(u)$ and of $R_6(u)$ we need to
understand the size and regularity of $b$. A-priori $b$ is only
defined on the diagonal $P_4$. However, in order to separate variables
it is convenient to extend it off diagonal in a favorable way.  The
next lemma is a more precise version of a similar result in
\cite{MR2353092}; the additional information is needed for the proof
of the local energy decay estimates in the next section.

\begin{proposition}
  Assume that $a \in S_\Lambda$ and $d\in S(a)$. Then there exist functions 
$b_4$ and $c_4$ such that 
\[ \begin{split} d(\xi_0) +d(\xi_1) - d(\xi_2) - d(\xi_3) 
= &  b_4(\xi_0, \xi_1,\xi_2,\xi_3) ( \xi_0^2 + \xi_1^2 - \xi_2^2 - \xi_3^2) 
\\ & + c_4(\xi_0,\xi_1,\xi_2, \xi_3) ( \xi_0 + \xi_1 - \xi_2 - \xi_3) 
\end{split} 
\]
which  for each
  dyadic
\begin{equation}\label{org}
 \lambda \leq \alpha \leq \mu, \qquad 
\xi_0 \sim \lambda,\quad  \xi_2\sim \alpha, \quad \xi_1, \xi_3 \sim \mu 
\end{equation}
 satisfy the size and regularity conditions
\begin{equation} \label{b4c4}
\begin{split}
| \partial_0^{\beta_0} \partial_1^{\beta_1}\partial_2^{\beta_2}
\partial_2^{\beta_3} 
 b_4(\xi_0,\xi_1,\xi_2,\xi_3)| \lesssim & \  a(\lambda) \alpha^{-1}\mu^{-1}
 \lambda^{-\beta_0} \alpha^{-\beta_2} \mu^{-\beta_1-\beta_3}
\\
| \partial_0^{\beta_0} \partial_1^{\beta_1}\partial_2^{\beta_2}
\partial_3^{\beta_3} 
 c_4(\xi_0,\xi_1,\xi_2,\xi_3)| \lesssim & \  a(\lambda) \alpha^{-1}
 \lambda^{-\beta_0} \alpha^{-\beta_2} \mu^{-\beta_1-\beta_3}
\end{split}
\end{equation}
with implicit constants   dependent on the $\beta_j$'s but
independent of $\lambda,\alpha,\mu$.

\label{pbbounds}\end{proposition}

\begin{proof}
We first note that  we have the formula
\begin{equation}\label{eq_decomp}  
\xi_0^2+\xi_1^2 -\xi_2^2 -\xi_3^2 = 
2(\xi_0 - \xi_2)(\xi_0-\xi_3)- (\xi_0 + \xi_1 - \xi_2 - \xi_3)(\xi_0 - \xi_1-\xi_2- \xi_3).  
\end{equation} 
In particular we obtain the factorization 
\[ \xi_0^2+\xi_1^2 -\xi_2^2 -\xi_3^2 = 
2(\xi_0 - \xi_2)(\xi_0-\xi_3) \qquad \text{ on }P_4  
\]
along with all versions of it due to the symmetries. 
It suffices to construct $b_4$ and $c_4$ locally in dyadic regions,
and then sum up the results using an appropriate partition of unity.
We consider several cases:

(a) $\lambda \ll \alpha \leq \mu$. Then $\xi_1, \xi_3 \sim \mu$ 
and $|\xi_0-\xi_2| \sim \alpha$. 
Then the extension of $b_4$ is
defined using the formula
\[
b_4(\xi_0,\xi_1,\xi_2,\xi_3) =   \frac{d(\xi_0)+d(\xi_1)-d(\xi_2)-d(\xi_3)}
{2 (\xi_0-\xi_2)(\xi_0-\xi_3)}.
\]
Its size and regularity properties are straightforward since
$|\xi_0 - \xi_2| \approx \alpha$ and $|\xi_0 - \xi_4| \approx \mu$.
By  \eqref{eq_decomp} we obtain 
\[ 
\begin{split} 
c_4 =   \frac{  d(\xi_0)+ d(\xi_1)-d(\xi_2)-d(\xi_3)- b_4 (\xi_0^2+ \xi_1^2- \xi_2^2- \xi_3^2) }{ \xi_0+ \xi_1- \xi_2- \xi_3} 
= b_4( \xi_0- \xi_1-\xi_2-\xi_3). 
\end{split} 
\]
The bounds for $c_4$  are also obvious.

(b) $\lambda \approx \alpha \ll \mu$. Then $|\xi_0|, |\xi_2|\sim \alpha$ and 
$|\xi_1|, |\xi_3| \sim \mu$. 
We define  the extension of $b_4$  using the formula 
\[
b_4(\xi_0,\xi_1,\xi_2,\xi_3) =   \frac{d(\xi_0)-d(\xi_2)}
{2(\xi_0 - \xi_2)(\xi_0-\xi_3)} +  \frac{d(\xi_1)-d(\xi_3)}
{2(\xi_3 - \xi_1)(\xi_3-\xi_0)}
\]
and, as above,  
\[ 
\begin{split}
c_4 =   \frac{d(\xi_0)-d(\xi_2)}
{2(\xi_0 - \xi_2)(\xi_0-\xi_3)}(\xi_0 - \xi_1-\xi_2-\xi_3)
 + \frac{d(\xi_1)-d(\xi_3)}
{2(\xi_3 - \xi_1)(\xi_3-\xi_0)}(\xi_3-\xi_0-\xi_1-\xi_2)
.\end{split} 
\]
Again the estimates are immediate.

(c) $\lambda \approx \alpha \approx \mu$. 
We define the  extension of $b_4$ by
\[
\begin{split} 
b_4(\xi_0,\xi_1,\xi_2,\xi_3) = & \ 
 \frac{d(\xi_1)-d(\xi_3)}
{2(\xi_3-\xi_0)(\xi_3-\xi_1)}
+ \frac{d(\xi_0)-d(\xi_1+(\xi_0-\xi_3))}{2(\xi_3-\xi_0)(\xi_3-\xi_1)}
\\ = &\ \frac{q(\xi_1,\xi_3)}{2 (\xi_0-\xi_3)}
+    \frac{q(\xi_1+(\xi_0-\xi_3),\xi_1+(\xi_0-\xi_3))}{2 (\xi_1-\xi_3)}
\end{split} 
\]
where $q$ is the smooth function
\[ 
q(\xi,\eta) = \frac{d(\xi)-d(\eta)}{\xi-\eta}. 
\]
Then 
\[
c_4 = b_4(\xi_3-\xi_0 - \xi_1-\xi_2)
\]
and  the estimates follow immediately.  
\end{proof}

Using the above lemma, the contribution of $E_1$ to the energy is easy
to control:

\begin{lemma}
Assume that $a \in S_\Lambda$.  Then 
\begin{equation}
|E_1(u)| \lesssim E_0(u) \Vert u \Vert_{H^{-1/2}_\Lambda}^2  
\end{equation}
\label{e1bd}\end{lemma}

\begin{proof} 
  The proof is easier than the proof of the more essential result
  below. Nevertheless it introduces some useful techniques. We expand the
  quadrilinear expression in the dyadic frequency components. Then
 for  $\lambda \le \alpha \le \mu$ we consider the expression
\begin{equation}
\label{e1di}
 \left|\int b_4(\xi_0,\xi_1, \xi_2, \xi_0+\xi_1- \xi_2)\hat u_{\lambda} \hat u_{\mu} \overline{ \hat  u_\alpha \hat u_\mu} d\xi_0 d\xi_1 d\xi_2 \right| 
\end{equation}
where the ranges of the $\xi_j$'s are as in \eqref{org}. In this range
we can express $b_4$ in the form
\begin{equation} 
 b_4(\xi_0,\xi_1,\xi_2, \xi_3) =  a(\lambda)  \alpha^{-1}\mu^{-1}        \eta( \xi_0/\lambda ,\xi_1/\mu, \xi_2/\alpha, \xi_3/\mu). 
\end{equation} 
where $\eta$ is compactly supported and smooth with bounds independent
of $\lambda, \alpha$ and $\mu$.

Due to \eqref{b4c4} we can expand $\eta$ into a rapidly convergent
Fourier series.  Since complex exponentials are products of complex
exponentials in the coordinates, and since multiplication by a complex
exponential of the Fourier transform corresponds to a translation in
$x$ space we can separate variables and reduce the problem to the case
when $b_4$ is simply replaced by $a(\lambda) \lambda^{-1}
\alpha^{-1}$. Then using Bernstein to bound the low frequency factors
in $L^\infty$ we obtain for the expression in \eqref{e1di}
\begin{equation*} 
\begin{split}
  \text{\eqref{e1di}} =  a(\lambda) \alpha^{-1} \mu^{-1} \int
  u_\lambda u_\mu \overline{u_\alpha u_\mu} dx  \lesssim 
  a(\lambda) \lambda^{\frac12} \alpha^{-\frac12} \mu^{-1} \Vert
  u_\lambda \Vert_{L^2} \Vert u_\alpha \Vert_{L^2} \Vert u_\mu
  \Vert_{L^2} \Vert u_{\mu} \Vert_{L^2}
\end{split} 
\end{equation*}
We estimate the high frequencies in $H^{-1/2}_\Lambda$ and sum with
respect to $\lambda$, $\alpha$ and $\mu$.
\end{proof}

\noindent The more difficult result we need to prove is
\begin{lemma} \label{r6} Assume that $a \in S_\Lambda$ is as
  above. Then we have
\begin{equation}
\left|\int_0^1 R_6(u) dxdt\right| \lesssim \| u\|_{X^a \cap X^a_{le}}^2  
\| u\|_{X_\Lambda^{-1/4}  \cap X^{-1/4}_{\Lambda,le}}^4 
\end{equation}
\end{lemma}

\begin{proof}
  We consider a full dyadic decomposition of all factors and express
  the above integral in the Fourier space as a sum of terms of the
  form

\[
K = \int_{0}^1 \int_{P_6} b_4(\xi_1,\xi_2,\xi_3,\xi_0) \hat
u_{\lambda_1}(\xi_1) \overline{\hat u_{\lambda_2}(\xi_2)} \hat
u_{\lambda_3}(\xi_3) P_{\lambda_0}(\overline{ \hat
  u_{\lambda_4}(\xi_4)} \hat u_{\lambda_5}(\xi_5) \overline{ \hat
  u_{\lambda_6}(\xi_6)}) d\xi dt
\]
where 
\[ 
P_6 = \{ \xi_1+\xi_3+\xi_5 = \xi_2 + \xi_4 +\xi_6 \}, \qquad \xi_0 =
\xi_1-\xi_2+\xi_3
\]
For each of the dyadic factors $u_{\lambda_j}$ we will only use the $U^2_\Delta$ 
norm, which is controlled as in \eqref{em}.

As in the previous lemma, since $b_4$ is smooth in each variable
on the corresponding dyadic scale we can expand it into a rapidly
convergent Fourier series.  This allows us to separate variables and reduce 
the problem to the case when $b_4$ has separated variables,
\[
b_4(\xi_1, \xi_2, \xi_3, \xi_0) =  a(\lambda) \alpha^{-1} \mu^{-1} 
\chi^1(\xi_1) \chi^2(\xi_2) \chi^3(\xi_3) \chi^0(\xi_0) 
\]
where $\chi_i$'s are unit size bumps which are smooth on the
respective dyadic scales and
$\{\lambda_1,\lambda_2,\lambda_3,\lambda_0\} = \{
\lambda,\alpha,\mu,\mu\}$. By definition $\chi^i(D)$ are bounded in
the $X_\lambda$ spaces, therefore we can discard $\chi^1$, $\chi^2$
and $\chi^3$ and incorporate $\chi^0$ into $P_{\lambda_0}$.


Similarly to above we may expand the Fourier multiplier $\eta$ for
$P_{\lambda_0}$ into a Fourier integral. We obtain for a Schwartz
function $\rho$
\[ 
\eta(\xi_0) =  \int \rho(y) e^{i \xi_0 y/\lambda_0} dy
\]
Since in the domain of integration for $K$ we have $\xi_0 =
\xi_1-\xi_2+\xi_3$, we can separate the exponential into three factors
which can be harmlessly absorbed into $u_{\lambda_1}$, $u_{\lambda_2}$
and $u_{\lambda_3}$. Thus we may as above simply drop $P_{\lambda_0}$
whenever we wish to do so.  We discard $P_{\lambda_0}$ if $\lambda_0$
is large. On the other hand, if it is smaller than $\lambda_4,
\lambda_5$ and $\lambda_6$ then we keep it to get better estimates.
The disadvantage in that case is that $P_0$ prevents us from using
bilinear $L^2$ estimates for factors located across $P_{\lambda_0}$.
To summarize, we have reduced the problem to the case when $K$ has the
form
\[
K = a(\lambda) \alpha^{-1} \mu^{-1}  \int_{0}^1 \int_{\R} 
u_{\lambda_1} \overline{u_{\lambda_2}} 
u_{\lambda_3} P_{\lambda_0}(\overline{ 
  u_{\lambda_4}}  u_{\lambda_5} \overline{ 
  u_{\lambda_6}}) dx dt
\]
where we have the additional freedom to discard $P_{\lambda_0}$ as
needed.  The placement of the complex conjugates is irrelevant here
therefore we may always assume without any restriction in generality that 
\[
\lambda_1 \le \lambda_2 \le \lambda_3, \qquad \lambda_4 \le \lambda_5 \le
\lambda_6.
\]
It is also convenient to reorganize the indices in an increasing fashion
\[
\{  \lambda_1, \lambda_2,\lambda_3, \lambda_4, \lambda_5,\lambda_6\}
= \{  \mu_1, \mu_2,\mu_3, \mu_4, \mu_5,\mu_6\}, \qquad \mu_5 \sim  \mu_6
\]
We also recall that $\lambda$, $\alpha$ and $\mu$ are given by the increasing 
rearrangement 
\[
\{  \lambda_0, \lambda_1, \lambda_2,\lambda_3\} = \{ \lambda,\alpha,\mu,\mu\}
\] 

The $X^{-\frac14}_\Lambda$ norms involve space and time localizations. We will
disregard those at first and consider the simpler question of estimating
the integral
\[
L =  \int_{\R^2} 
u_{\lambda_1} \overline{ u_{\lambda_2}} 
u_{\lambda_3} P_{\lambda_0}(\overline{ 
  u_{\lambda_4}}  u_{\lambda_5} \overline{ 
  u_{\lambda_6}}) dx dt
\]
in terms of the $U^2_\Delta$ norm of each factor,
\begin{equation}\label{jest}
|L| \lesssim C_L \prod_{j=1}^6 \|u_{\lambda_j}\|_{U^2_\Delta},
\qquad C_L = 
C_L(\lambda_0,\lambda_1,\lambda_2,\lambda_3,\lambda_4,\lambda_5,\lambda_6).
\end{equation}
In all cases $C$ will have a polynomial dependence on the $\lambda$'s
and a zero order homogeneity. 
Before we set to the task of estimating $C$ in all the cases,
we consider the simpler question of the transition from the estimate
for $L$ to the estimate for $K$. Precisely, we claim that \eqref{jest}
implies that 
\begin{equation}\label{iest}
|K| \lesssim C_K \prod_{j=1}^2 \|u_{\mu_j}\|_{X^a}
\prod_{j=3}^6 \|u_{\mu_j}\|_{X^{-\frac14}_\Lambda \cap X^{-\frac14}_{\Lambda,le}},
\ \  C_K = \frac{a(\lambda)}{\sqrt{a(\mu_1)a(\mu_2)}} 
\frac{1}{\alpha \mu} (\mu_3 \mu_4 \mu_5 \mu_6)^{\frac14} (\mu_1 \mu_2)^{\frac12} 
 C_L
\end{equation}
Compared to $C_L$, the constant $C_K$ contains the additional trivial
frequency factors coming from the Sobolev regularity, plus the more
interesting factor $(\mu_1 \mu_2)^\frac12$ coming from the time
interval summation. For $C_K$ we want to have $C_K \leq 1$, plus some
additional off-diagonal decay to allow for the summation with respect
to all $\lambda_j$'s.  We remark that since $s = -\frac14$, $C_K$ has
homogeneity zero if $a$ is homogeneous, therefore we do not have room
for any losses.

We now prove that \eqref{jest} implies \eqref{iest}. For convenience
we simply omit the prefactor $a(\lambda)/(\alpha \mu)$ in $K$, which
plays no role here.  We decompose each factor $u_{\lambda_j}$ in space
on the unit scale and in time according to the $\delta t_{\lambda_j} =
\lambda_j^{-1}$ scale, while preserving the frequency localization:
\[
u_{\lambda_j} = \sum_{|I_j|= \lambda_j^{-1}} \sum_{k_j \in \Z} u_{\lambda_j,k_j}^{I_j},
\qquad 
u_{\lambda_j,k_j}^{I_j}=
\chi_{I_j}\tilde P_{\lambda_j}
(\chi_{k_j} u_{\lambda_j})
\]
Then $K$ is decomposed into
\[
K = \sum_{I_j \text{nested}}^{|I_j|= \lambda_j^{-1}} \sum_{k_j \in \Z}
K((I_j)_{j=1,\cdots,6},(k_j)_{j=1,\cdots,6}),
\]
where
\[
K((I_j),(k_j)) = \int_{0}^1 \int_{\R} 
u_{\lambda_1,k_1}^{I_1} \overline{ u_{\lambda_2,k_2}^{I_2}} 
u_{\lambda_3,k_3}^{I_3} P_{\lambda_0}(\overline{ 
  u_{\lambda_4,k_4}^{I_4}}  u_{\lambda_5,k_5}^{I_5} \overline{ 
  u_{\lambda_6,k_6}^{I_6}}) dx dt
\]
For these components we claim that we have
\begin{equation}\label{Jik}
|K((I_j),(k_j))| \lesssim (1+ \max |k_i-k_j|)^{-N} 
C_L \prod_{j=1}^6 \|\chi_{k_j} u_{\lambda_j} \|_{U^2[I_j]}
\end{equation}
If $\max |k_i-k_j| \lesssim 1$ then this follows directly from \eqref{jest}.
Otherwise we further decompose $K((I_j),(k_j))$
as 
\[
K((I_j),(k_j)) = \sum_k   \int_{0}^1 \int_{\R} \chi_{k}
u_{\lambda_1,k_1}^{I_1} \overline{ u_{\lambda_2,k_2}^{I_2}} 
u_{\lambda_3,k_3}^{I_3} P_{\lambda_0}(\overline{ 
  u_{\lambda_4,k_4}^{I_4}}  u_{\lambda_5,k_5}^{I_5} \overline{ 
  u_{\lambda_6,k_6}^{I_6}}) dx dt
\]
For each $k \in \Z$ we can find some $j$ so that $|k-k_j| \gtrsim
\max\{k_i-k_j\}$.  To keep the notations simple let us take
$j=1$. Then we apply the bound \eqref{tails} to $\chi_{k}
u_{\lambda_1,k_1}^{I_1}$; this shows that
\[
\| P_{\mu} \chi_{k} u_{\lambda_1,k_1}^{I_1}\|_{U^2_\Delta} \lesssim 
\lambda_1^{-N} \mu^{-N} (\max\{k_i-k_j\}+ |k-k_1|)^{-N}\|\chi_{k_1} u_{\lambda_1} \|_{U^2_\Delta[I_1]}
\]
Hence \eqref{Jik} follows by applying \eqref{jest} to each of the terms in the 
above sum after a summation with respect to $\mu$ and $k$.

We obtain the bound for $K$ by summing \eqref{Jik} over the nested
intervals $I_j$ and $k_j$. We switch the frequencies to the $\mu_j$
notation. Using the fact that $\mu_5 = \mu_6$ and therefore $I_5=I_6$,
by Cauchy-Schwarz we obtain
\[
\begin{split}
|K| \lesssim & \ \sum_{\{k_i\}} (1+ \max |k_i-k_m|)^{-N} \prod_{j=1}^4 \max_{I_j} 
\|\chi_{k_j} u_{\mu_j} \|_{U^2_\Delta [I_j]} 
\prod_{j=5}^6 \left(\sum_{I_j} 
\|\chi_{k_j} u_{\mu_j} \|_{U^2_\Delta [I_j]}^2 \right)^\frac12
\\ 
  \lesssim & \ \prod_{j=1}^2   \left( \sum_{k} \max_{I_j} \|\chi_{k} u_{\mu_j} \|_{U^2_\Delta [I_j]}^2 \right)^{\!\!\! \frac12}
\!\! \prod_{j=3}^4 \sup_{k} \max_{I_j} \|\chi_{k} u_{\mu_j} \|_{U^2_\Delta [I_j]}
\prod_{j=5}^6 \sup_{k}  \left(\!\!\sum_{I_j} 
\|\chi_{k} u_{\mu_j} \|_{U^2_\Delta [I_j]}^2 \right)^{\!\!\! \frac12}
\\ 
  \lesssim & \ \prod_{j=1}^2   
\left( \sum_{I_j} \sum_{k}  \|\chi_{k} u_{\mu_j} \|_{U^2_\Delta [I_j]}^2 \right)^\frac12
\prod_{j=3}^6 \sup_{k}  \left(\sum_{I_j} 
\|\chi_{k} u_{\mu_j} \|_{U^2_\Delta [I_j]}^2 \right)^\frac12
\\
 =  & \   (a(\mu_1)a(\mu_2))^{-\frac12} (\mu_3 \mu_4 \mu_5\mu_6)^\frac14
(\mu_1\mu_2)^{\frac12}
 \prod_{j=1}^2   
\left( \sum_{I_j}  \| u_{\mu_j} \|_{X^a}^2 \right)^\frac12
\prod_{j=3}^6  \| u_{\mu_j} \|_{X^{-\frac14}_{le}} 
\end{split}
\]
Thus \eqref{iest} is proved. It remains to estimate
the constant $C_L$ in \eqref{jest}.  We need to distinguish two cases:

{\bf Case A.} $\lambda_0 \geq \mu_2$.  In this case we must have 
$\lambda \geq \mu_1$, $\alpha \geq \mu_2$ and $\mu \geq \mu_3$.
We claim the following bound
\begin{equation} \label{largel0}
C_L \lesssim (\mu_1 \mu_3)^{\frac12} \mu_6^{-1} 
\end{equation}
This is not optimal in many cases, but it suffices for our purposes.
In particular by \eqref{iest} it implies that
\[
C_K \lesssim  \frac{(\mu_1 \mu_3)^{\frac12}}{ \mu_6}  \frac{a(\mu_1)} {\mu_2 \mu_3}
(\mu_1 \mu_2)^\frac12 (a(\mu_1)a(\mu_2))^{-\frac12}
\mu_3^\frac14 \mu_4^\frac14 \mu_6^\frac12
\leq   \frac{\mu_1 \sqrt{a(\mu_1)}}{\mu_2 \sqrt{a(\mu_2)}} 
\frac{\mu_2^\frac14}{\mu_6^\frac14} \leq 1
\]
Note that we have rapid decay off the ``diagonal''
$\lambda_0 = \mu_1 = \mu_2=\mu_3 = \mu_4 = \mu_5 = \mu_6$,
which suffices for the dyadic summation.

To prove \eqref{largel0} we drop $P_{\lambda_0}$ and we consider three subcases:

{\bf Case A1.} $\mu_3 \ll \mu_6$. Then we can use two bilinear $L^2$ 
and two Bernstein to obtain the stronger bound
\[
C_L \lesssim (\mu_1 \mu_2)^{\frac12} \mu_6^{-1} 
\]

{\bf Case A2.} $\mu_1 \ll \mu_3 \sim \mu_6$. Then we we can only use a single
bilinear $L^2$ bound, one Bernstein and three $L^6$ bounds to get
\[
C_L \lesssim \mu_1^{\frac12} \mu_6^{-\frac12} 
\]
which still implies \eqref{largel0}.

{\bf Case A3.} $\mu_1 \sim \mu_6$. Then we simply use six $L^6$ bounds
to show that $C_K \lesssim 1$.

{\bf Case B.} $\lambda_0 \ll \mu_2$.  In this case we claim that 
the following bound holds:
\begin{equation} \label{smalll0}
C_L \lesssim \lambda_0  \mu_4^{-\frac12} \mu_6^{-\frac12} 
\end{equation}
To see that this suffices we consider two cases. If $\lambda_0 \ll \mu_1$
then we have $\lambda = \lambda_0$, $\alpha \geq \mu_1$ and $\mu \geq \mu_3$.
Then by \eqref{iest} we obtain
\[
C_K \lesssim  \frac{\lambda_0}{ \mu_4^{\frac12} \mu_6^{\frac12} }  \frac{a(\lambda_0)} {\mu_1 \mu_3}
(\mu_1 \mu_2)^\frac12 (a(\mu_1)a(\mu_2))^{-\frac12}
\mu_3^\frac14 \mu_4^\frac14 \mu_6^\frac12
\leq   \frac{\lambda_0 a(\lambda_0)}  
{\mu_1 a(\mu_1)}
 \frac{ \sqrt{\mu_1a(\mu_1)}}{ \sqrt{\mu_2 a(\mu_2)}} 
\frac{\mu_2^\frac14}{\mu_4^\frac14} \leq 1
\]
 On the other hand if $ \mu_1 \leq \lambda_0 \ll \mu_2$ then
we have $\lambda \geq \mu_1$, $\alpha \geq \lambda_0$ and $\mu \geq
\mu_3$. Then by \eqref{iest} we obtain
\[
C_K \lesssim  \frac{\lambda_0}{ \mu_4^{\frac12} \mu_6^{\frac12} }  \frac{a(\mu_1)} {\lambda_0 \mu_3}
(\mu_1 \mu_2)^\frac12 (a(\mu_1)a(\mu_2))^{-\frac12}
\mu_3^\frac14 \mu_4^\frac14 \mu_6^\frac12
\leq 
 \frac{ \sqrt{\mu_1a(\mu_1)}}{ \sqrt{\mu_2 a(\mu_2)}} 
\frac{\mu_2^\frac14}{\mu_4^\frac14} \leq 1
\]
In both cases we have decay off the expanded diagonal $\lambda_0 =
\mu_1 = \mu_2=\mu_3 = \mu_4$, $\mu_5 = \mu_6$, which still suffices
for the dyadic summation.

It remains to prove the bound \eqref{smalll0}. If $\lambda_0 \ll
\mu_2$ then we must have $\lambda_0 \ll \lambda_2= \lambda_3$ and
$\lambda_0 \ll \lambda_5=\lambda_6$. By symmetry we can assume that 
$\lambda_6 = \mu_6$. Then we have two cases to consider:

{\bf Case B1.} $\lambda_{3} \gtrsim \mu_4$. Applying twice \eqref{tril2}
we have 
\[
\|P_{\lambda_0} (u_{\lambda_1}\bar u_{\lambda_2}u_{\lambda_3})\|_{L^2} \lesssim 
\lambda_0^\frac12 \lambda_3^{-\frac12} \prod_{j=1^3}
 \|u_{\lambda_j}\|_{U^2_\Delta}, \qquad \|P_{\lambda_0} (u_{\lambda_4}\bar u_{\lambda_5}u_{\lambda_6})\|_{L^2} \lesssim 
\lambda_0^\frac12 \lambda_6^{-\frac12} \prod_{j=4^6}
 \|u_{\lambda_j}\|_{U^2_\Delta}
\]
which imply  \eqref{smalll0}.

{\bf Case B2.}$\lambda_{3} \ll \mu_4$. Then the frequencies must be ordered as
follows:
\[
\lambda_1 \lesssim \lambda_2 = \lambda_3 \ll \lambda_4 \lesssim
\lambda_5 = \lambda_6
\]
The key observation here is that, regardless of the presence of $P_{\lambda_0}$,
the multilinear interaction in $L$ is nonresonant, i.e. at least one 
of the factors must have high modulation $\gtrsim \lambda_4 \lambda_6$.
This is similar to the proof of \eqref{ineq:C}. For the following argument 
it does not matter which is the high factor modulation factor. To fix the
notations we assume this is the $\lambda_4$ factor; this is actually 
the worst case. Then we write
\[
\begin{split}
|L| \lesssim & \ 
\lambda_0 \| u_{\lambda_1} u_{\lambda_2} u_{\lambda_3}\|_{L^4 L^1}
\| Q_{> \lambda_4 \lambda_6} 
u_{\lambda_4} u_{\lambda_5} u_{\lambda_6}\|_{L^\frac{4}3 L^1}
\\
\lesssim & \ 
\lambda_0 \| u_{\lambda_1}\|_{L^\infty L^2} 
\| Q_{> \lambda_4 \lambda_6} 
u_{\lambda_4}\|_{L^2} \prod_{j=2,3,5,6} \|u_j\|_{L^8 L^4}
\\
\lesssim & \ 
\lambda_0 (\lambda_4 \lambda_6)^{-\frac12} \prod_{j=1}^6 \|u_j\|_{U^2_\Delta}
\end{split}
\]
This gives \eqref{smalll0} in this case, and concludes the proof of
the lemma.

\end{proof}

\section{Local energy decay}

In this section we consider the weighted local energy decay estimates
for \eqref{nls}. Our main goal is to prove Proposition~\ref{p:le},
which is the local energy counterpart of Proposition~\ref{p:en}
in the previous section. Proposition~\ref{p:le}, together with 
 Proposition~\ref{p:en}, will be used in the last section to derive the 
second part of Proposition~\ref{penergy}, namely the bound \eqref{le-nonlin}.
To keep the argument as simple as possible, in this section 
we only consider the extreme case $s = -\frac14$. The benefit 
of doing this is that at $s = \frac14$ we can work with the same
unit spatial scale for all frequencies.

Let $\phi$ be an odd smooth function whose derivative has the form
$\phi' = \psi^2$ where $\psi$ is positive,
with rapidly decaying and with Fourier transform supported in $[-1,1]$. Let
$a$ be as in the previous section.  We define an odd monotone smooth
function $\tilde a \in S_\Lambda(a)$ by
\[ 
\tilde a = \left\{ \begin{array}{ll} a(\xi) & \text{ if } \xi>\Lambda \\
                                    -a(\xi) & \text{ if } \xi < -\Lambda \\
\Lambda^{-1} \xi a(\xi) & \text{ if } |\xi| < \Lambda/2
\end{array} 
\right. 
\]
and consider the indefinite quadratic form 
\[ 
\tilde E_0(u) = \frac12   \int (\phi \tilde a(D)+ \tilde a(D) \phi) u \bar u dx. \]
A small modification of the calculation of the previous section 
gives
\[ 
\frac{d}{dt} \tilde E_0(u) =   \tilde R_2(u) \pm \tilde R_4(u) 
\]
where 
\[ 
\begin{split} 
\tilde R_2(u) = &  i ( \langle (\phi \tilde a(D)+ \tilde a(D) \phi)  u_{xx}, u \rangle - 
\langle (\phi \tilde a(D)+\tilde a(D) \phi)  u, u_{xx}  \rangle
\\ =  & 
  \langle (\phi' \tilde a(D)+ \tilde a(D) \phi')  D  u , u  \rangle
   +   \langle (\phi' \tilde a(D)+ \tilde a(D) \phi')    u ,D  u  \rangle
\end{split} 
\]
and
\[ 
\tilde R_4(u) =  2 \Re
\langle i (\tilde a(D) \phi+ \phi \tilde a(D))u, |u|^2 u\rangle  
\]
The term $\tilde R_2$, which was zero in the computation of the
previous section, has a positive principal symbol and will be used 
to measure the local energy.

 We now turn our attention to the quadrilinear form $\tilde R_4$.
In the Fourier space we represent this term in the form
\[
\begin{split}
\tilde R_4(u) = &\ \int_{\R} \phi(x) e^{i x \xi} \int_{P_\xi} 
i ( \tilde a(\xi_0-\xi) + 
\tilde a(\xi_0))  \hat u (\xi_0)  \hat u (\xi_1)
\overline{ \hat u (\xi_2) \hat u(\xi_3)}  d\xi_i d\xi dx 
\\ = &\ \int_{\R} \phi(x) e^{i x \xi} \int_{P_\xi} 
i \tilde a_4 (\xi_0,\xi_1,\xi_2,\xi_3)  \hat u (\xi_0)  {\hat u} (\xi_1)
\overline{ \hat u (\xi_2)     \hat u(\xi_3) } d\xi_i d\xi  dx
\end{split} 
\]
where 
\[
P_\xi = \{ \xi_0 + \xi_1 - \xi_2 - \xi_3 = \xi \}
\]
and the symbol $\tilde a_4$ is obtained by symmetrizing the symbol
$\tilde a(\xi_0-\xi) + \tilde a(\xi_0)$. We can view this as a
function of $\xi_0$ with a smooth dependence on the parameter $\xi$,
which is invariant with respect to the symmetrization. Here we only
need uniformity with respect to $\xi$ in a compact set $[-1,1]$.
Hence we can apply Lemma~\ref{pbbounds}, keeping $\xi$ as a parameter,
in order to represent the symbol $\tilde a_4$ in the form
\[
\tilde a_4 = \tilde b_4 (\xi_0^2 + \xi_1^2 - \xi_2^2 - \xi_3^2) + 
\tilde c_4  (\xi_0 + \xi_1 - \xi_2 - \xi_3) =
 \tilde b_4 (\xi_0^2 + \xi_1^2 - \xi_2^2 - \xi_3^2) + 
\tilde c_4  \xi
\]
where $\tilde b_4$ and $\tilde c_4$ are viewed as functions of
$\xi_0$, $\xi_1$, $\xi_2$, $\xi_3$ and $\xi$
which are smooth in $\xi$ in a compact set and are smooth on dyadic
scales $|\xi_j| \sim \lambda_j$ and have size
\[
\tilde b_4 \sim a(\lambda) \alpha^{-1} \mu^{-1}, \qquad 
\tilde c_4 \sim a(\lambda) \alpha^{-1}, \qquad
 \{\lambda_0,\lambda_1,\lambda_2,\lambda_3\} = \{
\lambda,\alpha,\mu,\mu \}, \quad \lambda \leq \alpha \leq \mu
\]
This leads to a decomposition 
\[
\tilde R_4(u) = \tilde B_4(u) + \tilde C_4(u)
\]
The $ \tilde C_4$ term is better behaved, as one can see in the
following integration by parts:
\[
\begin{split}
 \tilde C_4(u)= & \int  \phi(x) e^{i x \xi}\xi \int_{P_\xi}
 i \tilde c_4(\xi_0,\xi_1, \xi_2 , \xi_3,\xi) 
  u(\xi_1)  u (\xi_2) 
\overline{u(\xi_3) u (\xi_4)} d\xi_1 d\xi_2 d\xi_3 d\xi dx
\\
= & - \int  \phi'(x) e^{i x \xi} \int_{P_\xi}
 \tilde c_4(\xi_0,\xi_1, \xi_2 , \xi_3,\xi)  u(\xi_0)  u (\xi_1) 
\overline{u(\xi_2) \bar u (\xi_3)}  d\xi_1 d\xi_2 d\xi_3  dx d\xi dx.
\end{split}
\]
We will estimate $\tilde C_4$ directly. For the $ \tilde B_4$ term, on
the other hand, we introduce an energy correction
\[ 
\tilde E_1(u) =  \int  \phi(x) e^{i x \xi} \int_{P_\xi}  
\tilde b_4(\xi_0,\xi_1, \xi_2 , \xi_3,\xi) u(\xi_0) u (\xi_1) 
\overline{u(\xi_2) u (\xi_3)} d\xi_1 d\xi_2 d\xi_3\, d\xi\,   dx\, . 
\]
Then 
\[
\begin{split}  
\frac{d}{dt} 
\tilde E_1(u) = \tilde B_4(u) \pm  \tilde R_{6}(u)   
\end{split} 
\]
where $\tilde R_{6}(u)$ is given by
\[
R_6(u) = 2\Re \int  \phi(x) e^{i x \xi} \int_{P_\xi} i 
b_4(\xi_0,\xi_1,\xi_2,\xi_3) \widehat{|u|^2u}(\xi_0) \hat u(\xi_1) \overline{\hat
u(\xi_2) \hat u (\xi_3)} d\xi_j d \xi dx
\]

With all the notations above, our full energy relation reads
\[
\frac{d}{dt} (\tilde E_0(u) \mp \tilde E_1(u)) = \tilde R_2(u) \pm 
\tilde C_4(u) \mp \tilde R_{6} (u)   
\]
where the choice of the signs depends on the focusing or defocusing 
character of the problem, and plays no role in our analysis.
Our goal is to use this relation to estimate the time integral of 
$\tilde R_2(u)$, which in turn controls the local energy. Integrating
between $0$ and $1$ we obtain
\begin{equation} \label{tr2bd}
\begin{split}
\int_0^1  \tilde R_2(u) dt = \left. \ 
 (\tilde E_0(u) \mp \tilde E_1(u))\right|_0^1  \mp 
 \int_0^1 \tilde C_4(u) dt 
\pm  \int_0^1  \tilde R_6(u) dt 
\end{split}
\end{equation}
Following the steps in the previous section we bound the 
terms on the right. We begin with $\tilde E_0$ and $\tilde E_1$:

\begin{lemma}
  Let $a \in S_\Lambda$ and $\tilde a$, $\phi$ as above. Then at fixed
  time we have
\begin{equation} 
\left|\tilde E_0(u) \right| \lesssim E_0(u)
\end{equation}
respectively 
\begin{equation}
|\tilde E_1(u)| \lesssim  E_0(u) \Vert u \Vert_{H^{-1/2}_\Lambda}^2  
\end{equation}
\end{lemma}

\begin{proof} The proof repeats the proof of Lemma~\ref{e1bd}.
One begins with a Littlewood-Paley decomposition. Separating variables
the symbols $\tilde a$ and $\tilde b_4$ can be replaced by their 
sizes for each dyadic piece. Once this is done, we observe
that, since $\phi$ is bounded and has a compactly supported Fourier 
transform,  it can be harmlessly included in either factor and discarded.
The proof is concluded as in Lemma~\ref{e1bd}.
\end{proof}

We continue with the bound for $\tilde C_4$:

\begin{lemma}
  Let $a \in S_\Lambda$ and $\tilde a$, $\phi$ as above. Then
\begin{equation} \label{tc4}
\left| \int_{0}^1  \tilde C_4(u) dt \right | 
\lesssim \| u\|_{X^a}^2 \| u \|_{X^{-1/4}_{\Lambda}\cap X^{-1/4}_{\Lambda,le}}^2 
\end{equation}
\label{l:c4}\end{lemma}

This result does not have a counterpart in the previous section, and
requires a complete proof. In order to keep the argument fluid we
postpone the proof for the end of the section.
Finally we have the bound for $\tilde R_6$:
\begin{lemma}
   Let $a \in S_\Lambda$ and $\tilde a$, $\phi$ as above. Then
\begin{equation} \label{tr6}
\left| \int_0^1  \tilde R_{6}(u)dt  \right| \lesssim  
\Vert u \Vert_{X^a}^2 \Vert u \Vert_{X^{-1/4}_{\Lambda,le}}^4 .
\end{equation}
\end{lemma}
\begin{proof}
The proof repeats the proof of Lemma~\ref{r6}. Since $\hat \phi$ 
has compact support, it does not affect any of the dyadic frequency 
localizations. Since $\phi$ is bounded, it does not affect any of the 
$L^p$ bounds. Finally, since $\phi$ is time independent, it does not affect
any of the nonresonance considerations in Case B2.
\end{proof}

Finally, we turn our attention to $\tilde R_2$.
Since $\phi' = \psi^2$, we can rewrite $\tilde R_2(u)$ in the form
\[
\tilde R_2(u) = \| \psi  (D \tilde a(D))^{\frac12} u\|_{L^2}^2 
+ \langle r^{2,w}(x,D) u,u \rangle
\]
where, by a slight abuse of notation, $( \phi' (D \tilde
a(D))^{\frac12}$ stands for the smooth odd square root, and the
operator $r^{2,w}$ accounts for the lower order terms, with its symbol
$r^2$ satisfying
\[ 
  \left|\partial_x^a  \partial_\xi^b r^2(x,\xi) \right| 
\lesssim  \langle x\rangle^{-N} \langle\Lambda^2+\xi^2\rangle^{-b/2} a(\xi)
 \] 
and hence it has a a negligible effect
\begin{equation}\label{tr2}
| \langle r^{2,w}(x,D) u,u \rangle| \lesssim E_0(u). 
\end{equation}
Thus we obtain 
\begin{lemma}
The quadratic form $\tilde R_2$ satisfies the bound
\begin{equation}
\| \psi  (D \tilde a(D))^{\frac12} u\|_{L^2}^2 \leq \tilde R_2(u)+
c  E_0(u).
\end{equation}
\end{lemma}
Taking into account the last four lemmas, we have proved
that
\[
 \int_{0}^1 \| \psi  (D \tilde a(D))^{\frac12} u\|_{L^2}^2 dt
\lesssim \sup_t \Vert u(t) \Vert_{H^a}^2(1+\|u(t)\|_{H^{-\frac12}_\Lambda}^2) 
+ \Vert u \Vert_{X^a}^2 ( \Vert u \Vert_{X^{-1/4}_{\Lambda}\cap X^{-1/4}_{\Lambda,le}}^2 
+ \Vert u \Vert_{X^{-1/4}_{\Lambda,le}}^4). 
\]
The right hand side is translation invariant but the left hand side is not.
Hence we can replace $\psi$ by $\psi(\cdot + x_0)$ and take the 
supremum over $x_0$. But some straightforward computations show that
\[
\sup_{j} \sum_{\lambda \geq \Lambda} \lambda^{-1} a(\lambda) \| \chi_j
\partial_x u_\lambda \|^2_{L^2([0,1] \times \R)} \lesssim  
\sup_{x_0} \| \psi(\cdot + x_0)  (D \tilde a(D))^{\frac12} u\|_{L^2([0,1] \times \R)}^2
\]
Hence we have proved the main result of this section:
 
\begin{proposition} \label{p:le}
Let $a \in S_\Lambda$. Then all $L^2$ solutions $u$ to \eqref{nls} 
satisfy the following bound in the time interval $[0,1]$:
\begin{equation}\label{mainle}
  \sup_{j} \sum_{\lambda \geq \Lambda} \lambda^{-1} a(\lambda) \| \chi_j\partial_x
  u_\lambda \|^2_{L^2} \lesssim \sup_t \Vert u(t) \Vert_{H^a}^2(1+\|u(t)\|_{H^{-\frac12}_\Lambda}^2) 
  + \Vert u \Vert_{X^a}^2 ( \Vert u \Vert_{X^{-\frac14}_{\Lambda}\cap X^{-\frac14}_{\Lambda,le}}^2 
  + \Vert u \Vert_{X^{-\frac14}_{\Lambda,le}}^4).
\end{equation}
\end{proposition}

\begin{proof}[Proof of Lemma~\ref{l:c4}]
We recall that 
\[
\begin{split}
 \tilde C_4(u)=- \int  \phi'(x) e^{i x \xi} \int_{P_\xi}
 \tilde c_4(\xi_0,\xi_1, \xi_2 , \xi_3)  u(\xi_0)  u (\xi_1) 
\overline{u(\xi_2)  u (\xi_3)}  d\xi_j  d\xi dx .
\end{split}
\]
We use a Littlewood-Paley decomposition for all factors, denoting the
corresponding frequencies by
$\lambda_1,\lambda_2,\lambda_3,\lambda_4$.  Since $\hat \phi$ has
compact support, we can organize the four frequencies as usual
$\{ \lambda_1,\lambda_2,\lambda_3,\lambda_4\} = \{\lambda,\alpha,\mu,\mu\}$
with $\lambda \leq \alpha \leq \mu$. Within each dyadic term the 
symbol $c_4$ has size $a(\lambda) \alpha^{-1}$. Hence we can separate
variables and reduce the problem to estimating the expressions 
\[
c_{ \lambda\alpha\mu} =  a(\lambda) \alpha^{-1}
\int_{0}^1 \int_\R  \phi'(x) u_\lambda \bar u_\alpha u_\mu \bar u_\mu dx dt 
\]
The position of the complex conjugates is of minor importance; above
we chose the most interesting case. The argument applies to the other
case without major changes.
  We split time into time intervals
$I$ of size $\mu^{-1}$; the interval summation is accomplished due to
the fact that we use the local energy norms.

{\bf Case 1:} If $\lambda \ll \mu$ then we expand 
each factor with respect to the $\chi_j$  partition of unity 
on the unit spatial scale and use two bilinear $L^2$
bounds plus H\"older's inequality to obtain
\[
\begin{split}
|c_{\lambda\alpha  \mu}| \lesssim & \
a(\lambda) \alpha^{-1} \mu^{-1} \sup_j  \sum_{I}  
\| \chi_I \chi_j  u_\lambda \|_{U^2}\| \chi_I \chi_j   u_\alpha \|_{U^2}
\| \chi_I \chi_j   u_\mu \|_{U^2}\| \chi_I \chi_j   u_\mu \|_{U^2} 
\\
\lesssim & \
a(\lambda)^\frac12 a(\alpha)^{-\frac12} \alpha^{-1} 
\|   u_\lambda \|_{X^{a}}\|   u_\alpha \|_{X^{a}}
\|  u_\mu \|_{X^{-\frac12}_{\Lambda,le}}\| u_\mu  \|_{X^{-\frac12}_{\Lambda,le}} 
\end{split}
\]
with an easy summation in $\alpha$, $\lambda$ and $\mu$.

{\bf Case 2:} If $\Lambda \ll \lambda \sim \mu $ then we use three $L^6$
bounds, one energy and one H\"older inequality in time. We obtain a constant
\[
a(\mu) \mu^{-\frac32}
\]
which is again more than we need for the summation. 

{\bf Case 3:} If $\Lambda \sim \lambda \sim \mu $
then we have an additional difficulty, because we can no longer use 
the full strength of the local energy decay. In this case 
we can assume that $\tilde c_4$ is constant (and nonzero) 
so there is no help from there. In the defocusing case
this  term comes with the right sign, as in the classical Morawetz
estimate. However, in the focusing case
we need to bound it, and there is a potential 
obstruction which is due to the existence of solitons.
Indeed, consider the soliton 
\[
u = Q_\sigma e^{- i\sigma^2 t}, \qquad \sigma = \Lambda^\frac12 
\]
where the scale for $\sigma$ is chosen so that this soliton has the largest mass
among the zero speed solitons of $H^{-\frac14}_\lambda$ size less than one.

Suppose $Qe^{it}$ is the normalized soliton. Then
by rescaling we produce frequency $\sigma$ solitons
\[
Q_\sigma e^{i \sigma^2 t}, \qquad Q_\sigma = \sigma Q(\sigma x)
\]
with mass $\sigma^\frac12$. Such a soliton has norm less than $1$ in
$H^s_\Lambda$ provided that
\[
\lambda \leq \Lambda^{2s}
\]
Now measure the same soliton in our space $X^s_{\Lambda,le}$.
We loose $\Lambda^{-2s}$ in the time interval summation. On the other 
hand we gain $\sigma/\Lambda$ because of the $\partial_x$ operator 
in the definition of $X^s_{\Lambda,le}$. 
Thus for all $\sigma$ as above we must have 
\[
\Lambda^{-2s} \frac{\sigma}{\Lambda} \lesssim 1
\]
This gives exactly the threshold $s = -\frac14$ which corresponds 
to $\sigma = \Lambda^\frac12$. Hence not only our full result 
(i.e. with large $\Lambda$ ) is false for $s < -\frac14$, but also
the $\partial_x$ operator in the definition of $X^s_{\Lambda,le}$
cannot be relaxed at all if $s = -\frac14$.

Then the low frequency part of the integral in the
lemma has the form
\[
\frac{a(\Lambda)}{\Lambda} \int_0^1 \int |Q_\sigma|^4 dxdt \sim 
\frac{a(\Lambda)}{\Lambda} \sigma^3 = a(\Lambda) \Lambda^{\frac12}
\]
which is a tight bound. This shows that $s = -\frac14$ is the actual
threshold for this lemma, and also that in proving the lemma 
in this case we need to be careful about the concentration scale
associated to the above soliton.

We consider a further dyadic decomposition 
\[
u_{\Lambda} = \sum_{\sqrt{\Lambda}\leq \lambda \leq \Lambda} u_\lambda
\]
where $u_{\sqrt{\Lambda}}$ also contains all the lower frequencies.
Then $c_{\Lambda\Lambda\Lambda}$ is decomposed into 
\[
c_{\Lambda\Lambda\Lambda} = \sum c_{ \lambda\alpha\mu}, \qquad 
 c_{ \lambda\alpha\mu}
 =  a(\Lambda) \Lambda^{-1}
\int_{0}^1 \int_\R  \psi^2 u_\lambda \bar u_\alpha u_\mu \bar u_\mu dx dt 
\]
where $\lambda \leq \alpha \leq \mu$ are in the range 
$[\sqrt{\Lambda},\Lambda]$. This is a slight abuse of notation,
since $\lambda$, $\alpha$ and $\mu$ here are in a different range from the 
one previously considered.
We look at several cases:

{\bf Case 3a:} If $\lambda \ll \mu$ then we can use two bilinear $L^2$
estimates to get
\[
|c_{ \lambda\alpha\mu}| \lesssim a(\Lambda) \Lambda^{-1} \mu^{-1}
\sum_{|I| = \Lambda^{-1}} \|  u_\lambda \|_{U^2_\Delta[I]}
\|  u_\alpha \|_{U^2_\Delta[I]} \|\psi u_\mu\|_{U^2_\Delta[I]} 
 \|\psi u_\mu\|_{U^2_\Delta[I]}
\]
For the first two factors we use the $X^{a}$ norm to get a uniform
bound with respect to $I$. For the second two we use the local energy
norm to gain $l^2$ summability with respect to $I$, but there is a
price to pay, namely a $\Lambda/\mu$ factor for each due to the
$\partial_x$ operator in the definition of $X^s_{\Lambda,le}$.  Using
Cauchy-Schwarz with respect to $I$ for the last two factors to obtain
\[
\begin{split}
  |c_{ \lambda\alpha\mu}| \lesssim & \ a(\Lambda) \Lambda^{-1}
  \mu^{-1} \frac{\Lambda^\frac12} {a(\Lambda)}
  \left(\frac{\Lambda}{\mu}\right)^2 \| u_\lambda \|_{X^{a}} \|
  u_\alpha \|_{X^{a}} \| u_\mu \|_{X^{-\frac14}_{\Lambda,le}} \| u_\mu
  \|_{X^{-\frac14}_{\Lambda,le}}
  \\
  = & \ \frac{\Lambda^{\frac32}}{\mu^3}\| u_\alpha \|_{X^{a}} \| u_\mu
    \|_{X^{-\frac14}_{\Lambda,le}} \| u_\mu
    \|_{X^{-\frac14}_{\Lambda,le}}
\end{split}
\]
where the factor ${\Lambda^\frac12}/ {a(\Lambda)}$ is due to the $L^2$
normalizations of the four factors.  Since $\mu \geq \Lambda^\frac12$
the above coefficient is less than $1$ and we have an easy summation
with respect to $\lambda$, $\alpha$ and $\mu$.

{\bf Case 3b:} If $\lambda \sim \mu$ then we cannot use bilinear $L^2$
bounds. To understand this difficulty we consider first the extreme case:

{\bf Case 3b(i):}  $\lambda \sim \mu \sim \Lambda^\frac12$.
Then we neglect the local energy norms and simply use the energy for each factor
combined with two Bernstein inequalities and Holder in time. We obtain
\[
|c_{ \mu\mu\mu}|\lesssim  a(\Lambda) \Lambda^{-1} \mu 
\|u_\mu\|_{L^\infty L^2}^4 \lesssim  \frac{\mu}{\Lambda^\frac12}
\|  u_\mu \|_{X^{a}}^2 \|  u_\mu \|_{X^\frac14_\Lambda}^2
\]
which is favorable exactly when $\mu = \Lambda^\frac12$.
We continue with the last case:

{\bf Case 3b(ii):} $\lambda \sim \mu \gg \Lambda^\frac12$.  We begin
with the frequencies $\xi_i$ for the four factors.  Due to the compact
frequency support of $\psi$, these are restricted to a unit
neighborhood of the set $P_0 = \{ \xi_0 + \xi_1 - \xi_2 - \xi_3 =
0\}$.  We consider the dyadic scale $100 \sigma \sim 100 + \max |\xi_i
-\xi_j|$. The idea is now to produce a decomposition of $\tilde C_4$
with respect to $\sigma$. To achieve that we begin with a
corresponding decomposition of $\R^4$. For each dyadic $\sigma \geq 1$
we consider the family $\mathcal Q_\sigma$ of dyadic cubes $Q$ with
side-length $\sigma$, indexed by their position $(k_0,k_1,k_2,k_3)$.
Then we consider a Whitney type partition of $\R^4$ with respect to
the distance to the diagonal $\{\xi_0=\xi_1 = \xi_2 = \xi_3\}$
\[
\mathcal Q^1 = \bigcup_{\sigma \geq 1}\mathcal Q^1_\sigma, \qquad 
\mathcal Q^1_\sigma = \mathcal Q^1 \cap \mathcal Q_\sigma
\]
where for $\sigma > 1$ the cubes in $\mathcal Q^1_\sigma$ are at distance 
$\sim 100\sigma$ from the diagonal, while the cubes in $\mathcal Q_1^1 $
are within distance $\lesssim 100$ of the diagonal. 
To this partition of $\R^4$ we associate a corresponding partition of unit 
\[
1 = \sum_{Q \in \mathcal Q_1}  \chi_Q(\xi_0,\xi_1,\xi_2,\xi_3)
\]
where $\chi_Q$ is smooth on the $\sigma $ scale for $Q \in \mathcal Q_\sigma$.
This is possible since each two neighboring cubes in $\mathcal Q^1$
have comparable size. The functions $\chi_Q$ do not have separated variables,
but we can separate variables as before, and, by a slight abuse of 
notation, assume that $\chi_Q$ has the form
\[
\chi_Q =\chi_\sigma^{k_0}(\xi_0) \chi_\sigma^{k_1}(\xi_1)\chi_\sigma^{k_2}(\xi_2)
\chi_\sigma^{k_3}(\xi_3)  
\]
for $Q \in \mathcal Q_\sigma^1$ at position $(k_0,k_1,k_2,k_3)$.
Thus we obtain a simultaneous quadrilinear decomposition
\[
c_{ \mu\mu\mu} = \sum_{1 \leq \sigma \leq \mu} c_{ \mu\mu\mu}[\sigma]:=
a(\Lambda) \Lambda^{-1} \sum_{\sigma \geq 1}
\sum_{(k_i) \in \mathcal K(\sigma)} \int_{0}^1 \int_\R  
\psi^2 P_{\sigma}^{k_0} u_\mu P_{\sigma}^{k_1} u_\mu\overline{P_{\sigma}^{k_2} u_\mu}
\  \overline{P_{\sigma}^{k_3} u_\mu} dx dt 
\]
where the positions $k_i$ are restricted to a range $\mathcal
K(\sigma)$ with the property that
\[
\begin{split}
\mathcal K(\sigma) \subset& \ \{|k_i - k_j| \lesssim 1000, \ \max |k_i -k_j| \geq 100,
\ |k_0+ k_1 - k_2 - k_3| \leq 10  \}, \qquad \sigma > 1,
\\
\mathcal K(\sigma) \subset & \ \{|k_i - k_j| \lesssim 1000, \
\ |k_0+ k_1 - k_2 - k_3| \leq 10  \}, \qquad \sigma = 1.
\end{split}
\]
For each integral we can apply either the argument in Case 3a or the argument
in case 3B(i). In both cases the summation with respect to $k_j$ 
is diagonal and causes no losses by the Cauchy-Schwarz inequality.

For the  Case 3a argument we split the quadrilinear form in two bilinear
products with $\sigma$ frequency separation to obtain 
\[
|  c_{ \mu\mu\mu}[\sigma]| \lesssim   
\frac{\Lambda^{\frac32}}{\sigma \mu^2}\|  u_\mu \|_{X^{a}}^2 \|  u_\mu \|_{X^\frac14_{\Lambda,le}}^2
\]
For the  Case 3b argument we use Bernstein in $\sigma$ frequency intervals
to obtain
\[
|  c_{ \mu\mu\mu}[\sigma]| \lesssim   \frac{\sigma}{\Lambda^\frac12} 
\|  u_\mu \|_{X^{a}}^2 \|  u_\mu \|_{X^\frac14_{\Lambda,le}}^2
\]
Combining the two we have 
\[
| c_{ \mu\mu\mu}[\sigma]| \lesssim \min\left\{
  \frac{\Lambda^{\frac32}}{\sigma
    \mu^2},\frac{\sigma}{\Lambda^\frac12} \right\} \| u_\mu
\|_{X^{a}}^2 \| u_\mu \|_{X^\frac14_{\Lambda,le}}^2
\]
We remark that the two factors balance when $\sigma = \Lambda/\mu \leq
\Lambda^\frac12$. In the second case the frequency separation is not
needed, therefore the most natural decomposition would be obtained by
restricting $\sigma$ to the range $ \Lambda/\mu \leq \sigma \leq
\Lambda$.

Since $\mu \geq \sigma$ and $\mu \geq \Lambda^\frac12$ it is now easy
to check that the above coefficient is at most $1$, with decay off the
diagonal $\sigma = \mu = \Lambda^\frac12$. The proof of the lemma is concluded.

\end{proof}

\section{Conclusion}

The final step in the paper is to use
Propositions~\ref{p:en}, \ref{p:le} in order to conclude the proof
of Proposition~\ref{penergy}. For $a \in S_\Lambda$ we can combine 
the results in Propositions~\ref{p:en},\ref{p:le} to obtain
\begin{equation}
\begin{split}
\sup_t \sum_{\lambda \geq \Lambda} a(\lambda) \|u_\lambda(t)\|_{L^2}^2
+ &\ \sup_j \sum_{\lambda\geq \Lambda} a(\lambda) \lambda^{-1} 
\|\chi_j^\lambda \partial_x u_\lambda(t)\|_{L^2}^2 \lesssim
\|u_0\|_{H^a}^2+ \|u\|_{L^\infty H^a}^2 \|u\|_{L^\infty H^{-\frac12}_\Lambda}^2
\\ & \ +   \Vert u \Vert_{X^a}^2 ( \Vert u \Vert_{X^{-\frac14}_{\Lambda}\cap X^{-\frac14}_{\Lambda,le}}^2 
  + \Vert u \Vert_{X^{-\frac14}_{\Lambda,le}}^4)
\end{split}
\end{equation}
For $\mu \geq \Lambda$ we apply this inequality for
the symbols $a_\mu \in S_\Lambda$ given by  
\begin{equation}  
a_{\mu}(\xi) = \mu^{-\frac12} 
( 1+\xi^2/\mu^2)^{-\frac14-\varepsilon} 
\end{equation} 
with small $\varepsilon$. Restricting the left hand side in the above
inequality to $\lambda = \mu$ we have
\[
\begin{split}
 \mu^{-\frac12} \|u_\mu\|_{L^\infty L^2}^2
+  \sup_j \mu^{-\frac32} 
\|\chi_j^\mu \partial_x u_\mu(t)\|_{L^2}^2 \lesssim & \ \|u_0\|_{H^{a_\mu}}^2+
\|u\|_{L^\infty H^{a_\mu}}^2 \|u\|_{L^\infty H^{-\frac12}_\Lambda}^2
 \\ & +   \Vert u \Vert_{X^{a_\mu}}^2 
( \Vert u \Vert_{X^{-\frac14}_{\Lambda}\cap X^{-\frac14}_{\Lambda,le}}^2 
  + \Vert u \Vert_{X^{-\frac14}_{\Lambda,le}}^4)
\end{split}
\]
Finally, we sum up with respect to dyadic $\mu \geq \Lambda$ to obtain
\[
\begin{split}
  \| u\|_{l^2L^\infty H^{-\frac14}_{\Lambda}}^2 + \| u\|_{l^2l^\infty
    L^2H^\frac14_{\Lambda}}^2 \lesssim & \
  \|u_0\|_{H^{-\frac14}_\Lambda}^2+ \| u\|_{l^2L^\infty
    H^{-\frac14}_{\Lambda}}^2\|u\|_{L^\infty H^{-\frac12}_\Lambda}^2\\
  & + \Vert u \Vert_{X^{-\frac14}_\Lambda}^2 ( \Vert u
  \Vert_{X^{-\frac14}_{\Lambda}\cap X^{-\frac14}_{\Lambda,le}}^2 +
  \Vert u \Vert_{X^{-\frac14}_{\Lambda,le}}^4)
\end{split}
\]
This implies both \eqref{en-nonlin} and \eqref{le-nonlin} 
and completes the proof of Proposition~\ref{penergy} for $s = -\frac14$.

\bibliographystyle{plain}

\end{document}